\documentclass[12pt]{article}
\usepackage{amsmath,amsthm,amssymb,amsfonts, array}
\usepackage{fancyhdr,fancyvrb}

\usepackage[letterpaper,left=2.5cm,right=2.5cm,top=2cm,bottom=3cm]{geometry}
\usepackage{graphicx}
\usepackage{centernot} 
\usepackage{comment}
\usepackage{hyperref}
\usepackage[svgnames]{xcolor} 
\usepackage{mathrsfs}
\usepackage{enumerate}
\usepackage{url}
\usepackage{xspace}
\usepackage{tikz,forest} 
\usepackage{subfig}   
\usepackage{booktabs}   
\usepackage{multirow,multicol}
\usepackage{tcolorbox} 
\usepackage{rotating}
\usepackage{float}
\usepackage [english]{babel}
\usepackage [autostyle, english = american]{csquotes}
\MakeOuterQuote{"} 

\newtheorem{theorem}{Theorem}[section]
\newtheorem{lemma}[theorem]{Lemma}
\newtheorem{cor}[theorem]{Corollary}
\newtheorem{prop}[theorem]{Proposition}
\newtheorem{conjecture}[theorem]{Conjecture}
\newtheorem{definition}[theorem]{Definition}

\title{A degree preserving delta wye transformation with applications to 6-regular graphs and Feynman periods}
\author{Shannon Jeffries and Karen Yeats\footnote{SJ was supported by an NSERC USRA.  KY is supported by an NSERC Discovery grant and by the Canada Research Chairs program.  Both authors thank Michael Borinsky and Oliver Schnetz for useful discussions and for providing the results of their computations.}}

\begin{document}
\maketitle

\begin{abstract}
    We investigate a degree preserving variant of the $\Delta$-Y transformation which replaces a triangle with a new 6-valent vertex which has double edges to the vertices that had been in the triangle.  This operation is relevant for understanding scalar Feynman integrals in 6 dimensions.  We study the structure of equivalence classes under this operation and its inverse, with particular attention to when the equivalence classes are finite, when they contain simple 6-regular graphs, and when they contain doubled 3-regular graphs.  The last of these, in particular, is relevant for the Feynman integral calculations and we make some observations linking the structure of these classes to the Feynman periods.  Furthermore, we investigate properties of minimal graphs in these equivalence classes.
\end{abstract}

\section{Introduction}

In graph theory the delta-wye or $\Delta$-Y transformation takes a triangle in a graph and replaces it with a new three-valent vertex incident to the three vertices of the triangle, the Y-$\Delta$ operation takes a three-valent vertex and reverses this process.  

The $\Delta$-Y and Y-$\Delta$ transformations have been used in converting electric circuits since at least 1899 \cite{Kcircuit} and are now a standard technique taught in circuit analysis textbooks with a particularly important application in three-phase power systems (see for example section 11.5 of \cite{JHJC}). The $\Delta$-Y and Y-$\Delta$ transformations are also very interesting and widely studied as pure graph theory.  The main question on these operations studied as pure graph theory is $\Delta$-Y reducibility, that is when can a graph be reduced to a single vertex or another small fixed graph of interest using $\Delta$-Y and Y-$\Delta$ transformations, series parallel reductions, and removal of loops and degree one vertices \cite{EdeltaY, Gthesis, TdeltaY, Wagner, Yu0, Yu}, including algorithmic concerns \cite{DYintro, GSA, Tr}, and variants with marked vertices known as roots or terminals \cite{Dthesis, ACGP}.
Generalizing the application in electric circuits, the $\Delta$-Y transformations are also important in statistical mechanics (see \cite{Opotts, Fpotts, PBstartriangle, AUYANG198957, Bstartriangle} and the references therein).  In this context they are usually known as \emph{star triangle relations}.  In quantum field theory these transformations are also important, either again known as star triangle relations, see for instance \cite{CDI, GSstartriangle, Brbig, Cmsc, borinsky2021graphical}, or in the form of the \emph{method of uniqueness} originating from \cite{Parisi:1972zm}, see also \cite{Kazakov:1983dyk, Gracey1992OnTE} for some important subsequent development and uses.

We come to the problem from the direction of the last three quantum field theory references.  There the focus is on understanding the mathematical structure of Feynman integrals emphasizing a particularly simple class of quantum field theories known as scalar field theories, and emphasizing methods for calculating and understanding the Feynman period, an integral that can be viewed as an important residue of the Feynman integral.  The techniques are predominantly algebraic and graph theoretic.  In particular key graph polynomials transform well under the $\Delta$-Y and Y-$\Delta$ transformations.  However, the graphs involved are predominantly regular graphs or almost regular graphs and so the $\Delta$-Y is not ideally suited to the situation.  

In \cite{borinsky2021graphical, BSphi3}, Borinsky and Schnetz study Feynman periods and Feynman integrals of $\phi^3$ graphs.  These are 3-regular graphs except possibly for at most three degree 2 vertices.  Their main technique, known as the method of graphical functions \cite{Sgraphfn}, uses graph transformations with up to three marked vertices to simplify the graph, building the result of the integral as the graph reduces.  When applied to $\phi^3$ graphs specifically it is useful to double all the edges of the graph to obtain a 6-regular graph and then perform graph transformations which are well behaved with respect to the graphical functions and which preserve 6-regularity on this 6-regular graph.  One of these operations is a 6-regular analogue of the $\Delta$-Y and Y-$\Delta$ transformations.  It is this operation that we will study in the present paper.

Specifically, we are interested in the operation which converts a triangle into a new 6-valent vertex connected by a double edge to each vertex of the original triangle, and in the reverse operation.  Borinsky and Schnetz show that this operation leaves the Feynman period unchanged, so this is an operation that is significant in quantum field theory, but we will study it primarily as a purely graph theoretic operation.  The reader does not need to know any quantum field theory to understand the paper, though the specific questions we are interested in remain motivated by the field theory and we will make some comments on this application and observed patterns in the Feynman periods.

\subsection{Roadmap}

The paper will proceed as follows.  In Section~\ref{subsec graph theory} we give the graph theory background including precise definitions of the original $\Delta$-Y and our new degree preserving version.  Section~\ref{sec: Feynman Periods} gives a brief overview of the parts of quantum field theory which lead to this operation.  This section is intended for graph theorists, but the reader who is not concerned with the underlying motivation or already knows some quantum field theory can skip it.  In Section~\ref{sec props} we will discuss how the degree preserving version of $\Delta$-Y interacts with standard graph theory notions such as planarity and cyclomatic number.  Some of these behave as for the usual $\Delta$-Y and some do not; we restrict our attention to those with some relevance to Feynman periods.  In Section~\ref{sec equiv classes} we move to our main question of interest, the nature of the equivalence classes of 6-regular graphs under the degree preserving version of the $\Delta$-Y operation.  We focus on when the equivalence classes are finite and when they are infinite in Sections~\ref{sec add ons} and \ref{sec finite infinite} culminating in Theorems~\ref{thm infinite} and \ref{thm finite} giving a characterization of when classes are finite and when they are infinite.  Section~\ref{sec doubled 3 reg} proceeds to consider the equivalence classes in the case of primary physical interest, doubled 3-regular graphs, and makes some observations on how the equivalence classes appear to relate to the Feynman periods calculated by Borinsky and Schnetz \cite{BSphi3}.  The main theorems of the previous section are particularly nice in this context as shown in Corollary~\ref{cor finite}.  The computational observations on the relation to Feynman periods suggest that the minimal graphs in these equivalence classes are particularly interesting.  Section~\ref{sec simple} considers the equivalence classes of simple graphs.  These are special both in that their equivalence classes are always finite, see Corollary~\ref{cor simple finite}, and in that simple graphs are minimal, see Theorem~\ref{thm minimal}.  Finally the paper concludes with a discussion of the code used to explore some of these ideas with a particular focus on some conjectures regarding computing minimal graphs, along with a discussion of open questions, and a return to the initial motivation, in Section~\ref{sec discussion}.

\section{Background and set up}

\subsection{Graph Theory}\label{subsec graph theory}

Although originally motivated by work with Feynman graphs in quantum field theory, as will be described in more detail in Section~\ref{sec: Feynman Periods}, this paper consists mostly of purely graph theoretical results and observations. As a result, it is essential that we define the vocabulary that we intend to use throughout this paper.
For us, a \textit{multigraph} may have multiedges but not loops, and a graph that does not have multiedges or loops will be called a \textit{simple graph}. The term \textit{graph} will refer to both multigraphs and simple graphs.  Further graph theory definitions will come from \cite{Dbook}, unless otherwise specified.

The cyclomatic number of a graph is defined as the dimension of the cycle space of the graph.  In physics this is known as the loop number.  It is also the first Betti number of the graph as a topological space.

The girth of a graph with at least one cycle is the minimum length of a cycle in the graph.

The Menger-Whitney Theorem will be useful to us.  It says that a simple graph $G$ is $k$-connected if and only if for every pair of vertices $a$ and $b$ in $G$, there exists at least $k$ independent $a-b$ paths in $G$.

Not directly relevant for our work, but relevant to much of the standard $\Delta$-Y literature is the notion of forbidden minors and minor closed classes of graphs.  One graph is a minor of another if the graph is the result of a (possibly empty) sequence of edge deletions and edge contractions of the other graph.  A class of graphs is minor closed if a graph being in the class implies all its minors are also in the class.  Planar graphs are famously a minor closed class of graphs.  As another example, linklessly embeddable graphs are those which can be embedded into 3-space without any cycles being linked.  This is also a minor closed class \cite{RSTlinkless}.  A famous result of Robertson and Seymour, the culmination of their graph minors project \cite{RSXX}, is that every minor closed graph class is defined by a finite set of forbidden minors.

\medskip

 The degree preserving $\Delta$-Y operation or $\Delta$-YY operation described in this paper is based on the delta-wye or $\Delta$-Y transform. As shown at the top of Figure \ref{fig:DeltaY-Transformations}, a \textit{$\Delta$-Y transformation} takes the delta shape with vertices $\{u, v,w\}$ and edges $\{uv, uw, vw\}$ and transforms it by adding a vertex $x$, removing the edge set $\{uv, uw, vw\}$, and replacing it with a wye shape with edge set $\{ux ,vx, wx\}$.  See \cite{DYintro} for further details.  We call the opposite transformation, starting with a wye and moving to a delta, a \textit{Y-$\Delta$ transformation}. This is shown at the bottom of Figure \ref{fig:DeltaY-Transformations}. Note that a wye is made up of a vertex with degree 3, its 3 neighbours, and the three edges connecting it to its neighbours. 

\begin{figure}[h]
    \centering
    \includegraphics{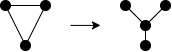}\\
    \includegraphics{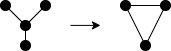}
    \caption{Top: A $\Delta$-Y transformation. Bottom: A Y-$\Delta$ transformation.}
    \label{fig:DeltaY-Transformations}
\end{figure}

From a graph theoretical perspective, the primary interest in Y-$\Delta$ and $\Delta$-Y transformations is the reducibility of graphs. For reducibility questions, in addition to the $\Delta$-Y and Y-$\Delta$ transformations, four additional operations are added that allow us to reduce graphs. We can delete a loop, delete a degree-one vertex and its edge, delete a degree-two vertex and its edges then add an edge connecting its neighbours, or replace a pair of parallel edges with a single edge. If a graph can be reduced to a single vertex using the $\Delta$-Y transformation, Y-$\Delta$ transformation, and these four additional operations, we say that it is \textit{$\Delta$-Y reducible}. At the core of this question of reducibility, we have an idea of equivalence classes of graphs under the $\Delta$-Y operation. Two graphs $G$ and $H$ are \textit{$\Delta$-Y equivalent} if there exists a series of $\Delta$-Y and Y-$\Delta$ transformations that transform $G$ into $H$. An equivalence class is a maximal set of graphs that are all equivalent to each other.  The Petersen family, the 7 graphs shown in Figure \ref{fig:Petersen} which form the $\Delta$-Y equivalence class of the Petersen graph, is a particularly famous example as they are the forbidden minors for linkless embedding \cite{RSTlinkless}.  
\begin{figure}
    \centering
    \includegraphics[width=12cm]{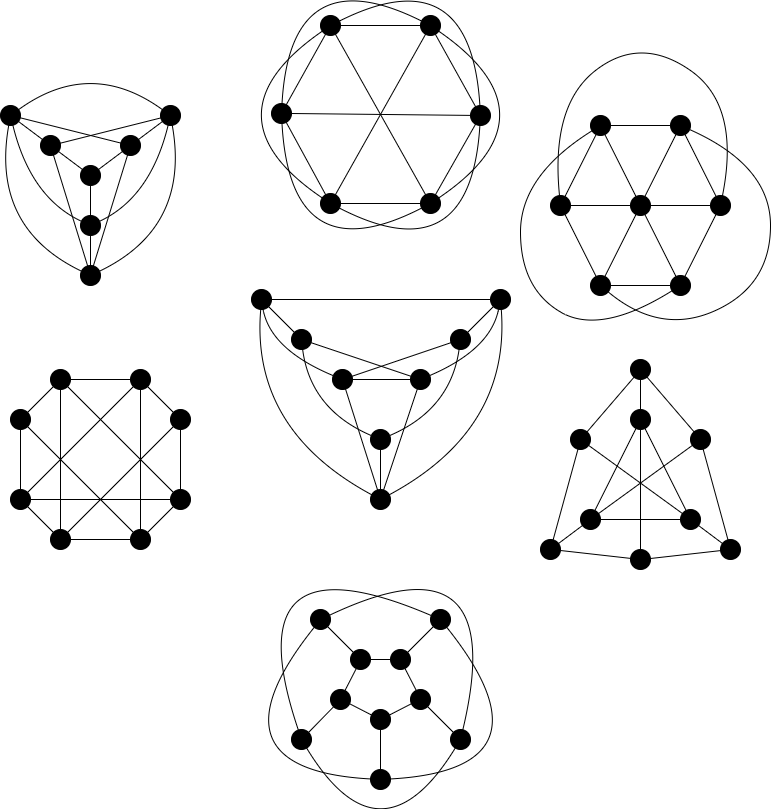}
    \caption{The Petersen family}
    \label{fig:Petersen}
\end{figure}
$\Delta$-Y equivalence classes allowed Yaming Yu to characterize the set of forbidden minors of $\Delta$-Y reducible graphs, and a similar concept is at the core of \cite{Yu}.

In this paper, we will discuss a variant of this $\Delta$-Y operation that preserves vertex degree. We will look at how it differs from the $\Delta$-Y operation and investigate its equivalence classes.

Let us now define the operations that we are interested in, which we will call the YY-$\Delta$ and $\Delta$-YY transformations. 
\begin{definition}
The \emph{YY-$\Delta$ transformation} is defined as follows.  Suppose we have a set of four vertices $\{u, v, w, x\}$, with the multiset of edges $\{ux, ux, vx, vx, wx, wx\}$ and where $x$ is 6-valent, that is there are no further edges incident to $x$.  Transform these vertices and edges by removing all of these edges and the vertex $x$ and adding the edges $\{uv, uw, vw\}$, creating a 3-cycle between $u$, $v$, and $w$. 

The \emph{$\Delta$-YY transformation} is the reverse operation, defined as follows.  Suppose we have vertices $\{u, v, w\}$ and edges $\{uv, uw, vw\}$.  Transform by removing all of these edges, adding a new vertex $x$ and new edges $\{ux, ux, vx, vx, wx, wx\}$. 

We use the term \emph{$\Delta$-YY operation} to refer to both transformations. These can be more concisely described as replacing a wye with a delta or a delta with a wye respectively, where a delta is depicted in Figure \ref{fig:delta} and a wye is depicted in Figure \ref{fig:doublewye}
\end{definition} 
Note that if we begin with a 6 regular graph and perform either $\Delta$-YY operation then the result is also a 6 regular graph.

These transformations get their names from the shapes these collections of vertices and edges resemble. A $\Delta$-YY transformation turns a delta shape into a Y shape of double edges, hence the double Y. A YY-$\Delta$ transformation turns a Y shape with all double edges into a delta shape.

\begin{figure}[h]
    \centering
    \begin{minipage}{.5\textwidth}
        \centering
        \includegraphics[width=2cm]{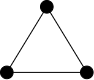}
        \caption{$\Delta$ Shape, also known as a \textit{delta}}
        \label{fig:delta}
    \end{minipage}%
    \begin{minipage}{0.5\textwidth}
        \centering
        \includegraphics[width=2cm]{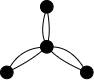}
        \caption{YY Shape, also known as a \textit{wye}}
        \label{fig:doublewye}
    \end{minipage}
\end{figure}

Properties of a graph which do not depend on the multiplicity of edges, such as Hamiltonicity, clearly behave in the same way under the $\Delta$-YY operations as under the classical $\Delta$-Y operations, while other properties behave differently, for instance trivially having an Eulerian tour is preserved under $\Delta$-YY operations but not under $\Delta$-Y operations.  We will discuss a few properties of interest in Section~\ref{sec props}.

The interplay between a multigraph and its underlying simple graph as well as the deltas and wyes themselves will be important for our study of $\Delta$-YY operations.  To this end we make the following definitions.
\begin{itemize}
    \item A \emph{connection} is a non-empty set of all the edges between two vertices, and each edge in a connection will be simply called an edge. Note that a connection with only one edge is allowed.
    \item The simplified version of a graph, $G^S$ is defined as follows. Given a multigraph $G$, we get $G^S$ by turning every connection of $G$ into a connection of size 1, giving us a simple graph.
    \item We say that two vertices are \emph{doubly adjacent} when the connection between them has size exactly 2.
    \item A \emph{double-edged graph} is a graph where every connection has size exactly 2. 
    \item A \emph{wye} (Figure \ref{fig:doublewye}) is the subgraph made up of vertex set $\{u, v, w, x\}$ and edge set $\{ux, ux, vx, vx, wx, wx\}$ that gets transformed in the YY-delta transformation. A \emph{simple wye} is a wye with edge set $\{ux, vx, wx\}$ instead of $\{ux, ux, vx, vx, wx, wx\}$ and will only be used when referring to the original $\Delta$-Y operation. In both the wye and the simple wye, $x$ cannot have any other edges incident to it.
    \item When referring to a wye made up of a vertex $x$ and its three neighbours, we say that $x$ is the \emph{centre} of the wye and its neighbours are the \emph{external vertices} of the wye.
\end{itemize}
  
\subsection{Feynman Periods in quantum field theory}
\label{sec: Feynman Periods}

Quantum field theory is the framework by which we can understand arbitrary numbers of interacting particles quantum mechanically.  It unifies quantum mechanics and special relativity.  When applied to the fundamental particles studied in high energy physics, quantum field theory has been enormously successful, giving some of the most precise correlations between prediction and experiment anywhere in science.  

Perturbative quantum field theory is the approach to quantum field theory where quantities of physical significance are computed by expanding in some small parameter.  The most important such expansion for us, which is also the most important historically, is the loop expansion by Feynman diagrams.  In this expansion, physical quantities are computed by an expansion which is indexed by certain graphs organized by increasing \emph{loop number} and where each graph contributes a \emph{Feynman integral} to the expansion.  The expansion parameter here is the \emph{coupling constant} or a power thereof.  Taking the special case where spacetime has 0 dimensions, this approach gives us a rigorous and powerful tool for graph counting, called 0-dimensional field theory.  See \cite{LZgraphs} for a graph theoretic exposition.  In other dimensions of spacetime, such as 4 or 6, the standard derivation of the loop expansion in Feynman diagram (using the path integral) is not mathematically rigorous, but as mathematicians we can take the Feynman diagrams, the Feynman integrals, and the loop expansion as the fundamental definitions and explore their mathematical properties.

Feynman graphs are graphs where the edges roughly represent particles and the vertices particle interactions.  Different quantum field theories give different types of edges (directed or undirected with different additional data carried along) and different types of vertices (allowable degrees and allowable combinations of edge types).  External edges represent particles entering or exiting the system and can be thought of as unpaired half edges or edges which go off to infinity since they are incident to only one vertex.  

For us the \emph{size} of a Feynman graph will be its loop number.  For practical quantum field theory computations a more meaningful notion of size relates to some useful measure of the magnitude of the Feynman integral; in particular for a given experiment only a certain domain of external momenta may be of interest and in this restricted domain some Feynman integrals may dominate while other give a negligible contribution even when they are of the same loop order.  However, for a mathematical consideration of the underlying algebraic structures we do not want to restrict to any particular configuration of external momenta and the most useful notion of size is simply the loop number.  Since we will be working with connected regular graphs, by Euler's formula the loop number, the number of vertices, and the number of edges are related.

Feynman integrals can be built out of the Feynman graph -- each edge and vertex makes a contribution to the integrand.  They are quite intricate and depend on many physical parameters such as masses of the particles and momenta of the particles coming into and going out of the system.  A simplified integral, known as the \emph{Feynman period} \cite{bkphi4, Brbig, Sphi4}, maintains much of the mathematical richness of the full Feynman integral while also still being of physical relevance as it is a particular residue of the full Feynman integral and gives the contribution of the Feynman graph to important things in quantum field theory such as the renormalization constants.

The Feynman period can be defined as follows.  Given a graph $G$, associate a variable $a_e$ to each edge $e$ and let the Kirchhoff polynomial of $G$ be 
\[
    \Psi_G = \sum_{T}\prod_{e\not\in T}a_e,
\]
where the sum is over all spanning trees $T$ of $G$.  Then the Feynman period is the following integral when it converges
\[
    \int_{a_e\geq 0}  \frac{da_2\cdots da_{|E(G)|}}{\Psi_G^{D/2}|_{a_1=1}},
\]
where $D$ is the dimension of spacetime.
This is a nice integral from a graph theoretic perspective because it is controlled by $\Psi_G$ which is a combinatorially defined polynomial.  The two main techniques for calculating Feynman periods also both have algebraic and graph theoretic flavours.  Denominator reduction and its generalizations, notably HyperInt\cite{PANZER2015148} and HyperlogProcedures\cite{Shyper}, integrate one edge at a time giving polynomials in terms of spanning forests in the denominator at each step and numerators in terms of multiple polylogarithms \cite{Brbig, Bmpl, Phyp, Pphd}.  Graphical functions \cite{Sgraphfn} build the integral one edge at time by working with a graph with three marked vertices, one of which is marked with $z$, the variable of the function.  The process is explicitly graph-based at every step.  Note that neither approach can compute all Feynman periods -- that would be much too much to ask -- but both of these techniques can also calculate some Feynman integrals, not just the periods.

Physically, the Feynman period is the contribution of the graph to the beta function of the theory.  From this or directly from the definitions, one can see that the Feynman period does not depend on the external momenta of the graph.  Alternately, breaking an internal edge we can see the periods as single scale integrals.  Despite this lack of, or very limited, dependence on the external parameters, Feynman periods are still useful in many quantum field theory calculations, see for example the introduction to \cite{BSphi3} and references therein.  They are also mathematically very interesting since they distill out key number theoretic content from the Feynman integral.

In view of the lack of dependence on external momenta, we can remove the external edges from the Feynman graph obtaining a graph in the sense of this paper.  These graph may have double edges or potentially higher multiple edges.  In principle there may be self-loops, but physically depending on the theory these either do not appear or their contribution factors out and in any case they do not contribute to the period, so we will lose nothing of interest by not allowing them.

Much of the work using both these integration approaches was done with graphs in $\phi^4$ theory in 4 dimensions.  In $\phi^4$ theory the Feynman graphs are 4-regular when counting the external edges, and so with the external edge removed they are 4-regular except for a small number of lower degree vertices. Furthermore an operation known as \emph{completion} lets us move to honestly 4-regular graphs.  Many special techniques were developed and graphical symmetries proved for this 4-regular case \cite{Sphi4}.

Recently, Borinsky and Schnetz used the graphical function technique to study $\phi^3$ theory in 6 dimensions \cite{borinsky2021graphical, BSphi3}.  This opened up a new world of graph symmetries that are period invariants in this 3-regular, 6-dimensional case.  Again, with completion we can move from almost 3-regular to actual 3-regular graphs.  To best describe the graph symmetries, Borinsky and Schnetz found it convenient to take the 3-regular Feynman graph and double all the edges resulting in a 6-regular graph.  In the context of their Feynman integral computations, edges have weights, and this process of doubling the edges corresponds to viewing each edge of weight 1 as two edges of weight $1/2$ and then considering the problem in terms of weight $1/2$ edges.  On this 6-regular graph, they discovered that the $\Delta$-YY operations were period invariants.  Consequently 6-regular graphs in the same $\Delta$-YY equivalence classes have the same Feynman period and hence understanding these classes better is of physical interest.  Particularly interesting is understanding the classes that contain a doubled 3-regular graph.  This is the motivation for the present work.

As well as the $\Delta$-YY, planar duals and small vertex and edge cuts give period identities in both the $\phi^4$ and $\phi^3$ cases, so it is valuable to understand how planarity and connectivity relate to the $\Delta$-YY operations.

\medskip

Finally, we should say something about the kinds of numbers which come up in Feynman periods.  
If $s_1, s_2, \ldots s_k$ are positive integers with $s_1 > 1$ then the multiple zeta value $\zeta(s_1, s_2, \ldots, s_k)$ is defined to be
\[
\zeta(s_1, s_2, \ldots, s_k) = \sum_{n_1>n_2>\cdots n_k>0}\frac{1}{n_1^{s_1}n_2^{s_2}\cdots n_k^{s_k}}.
\]
Multiple zeta values generalize special values of the Riemann zeta function and are studied for their algebraic and number theoretic properties.  A lot of their study is in fact combinatorics of words. For graphs of low loop order, Feynman periods are expressible as rational linear combinations of products of multiple zeta values.  This does not remain true for all graphs as the loop order increases, elliptic polylogarithms appear and ultimately even more exotic things.  By Mnev universality, in some sense everything ultimately appears \cite{BrBe} at least when we consider all graphs.  For Feynman graphs in renormalizable field theories, the periods go beyond multiple zeta values \cite{BrS, Doreg}, but they do not appear to be universal, for instance elliptic curves do not seem to appear though other modular forms do \cite{Sgeo}. All the periods we will consider explicitly in Section~\ref{sec doubled 3 reg} are expressible in terms of multiple zeta values.  The only other thing we will need about multiple zeta values is that the weight of a multiple zeta value is the sum $s_1+s_2+\cdots s_k$ and the weight of a product of multiple zeta values is the sum of the weights of the factors.  This notion of weight should be taken to be defined on the list $(s_1, s_2, \ldots, s_k)$, though up to standard conjectures it relates to the transcendental weight of the multiple zeta values as numbers.

The ideas of this section are outlined with some more details and references in \cite{Ybook}.

\section{$\Delta$-YY operations and graph properties}\label{sec props}

We will proceed by looking at how some common graph properties are preserved and which are changed after a $\Delta$-YY or YY-$\Delta$ transformation. 

First some useful elementary observations, 
the $\Delta$-YY operation preserves vertex degree unlike the $\Delta$-Y operation. The operation does not preserve vertex number as a $\Delta$-YY transformation adds one vertex and a YY-$\Delta$ transformation deletes one vertex, just as with the original $\Delta$-Y operation. The $\Delta$-YY operation does not preserve edge number as a $\Delta$-YY transformation deletes 3 edges and adds 6 for a net difference of an additional 3 edges and a YY-$\Delta$ transformation results in a graph with 3 fewer edges. This differs from the $\Delta$-Y operation which preserves edge number. 

A $\Delta$-YY transformation makes no difference to the length of a path or a cycle of $G^S$ unless an edge of the delta is in the path or cycle, in which case the length of said path or cycle increases by one. Similarly, a YY-$\Delta$ transformation has no effect on the length of a path or cycle of $G^S$, unless the centre of the Y is a vertex in the path or cycle, in which case the length of said path or cycle decreases by one. Theses simple facts become incredibly useful when proving results about other properties of the $\Delta$-YY operation.  

\subsection{Planarity}
Planarity behaves under $\Delta$-YY operations as it does under $\Delta$-Y operations, however planarity is important for Feynman periods so we will take the time to lay out this behaviour here.

\begin{prop}

\begin{enumerate}
    \item Given a planar graph $G$, planarity is preserved in a $\Delta$-YY transformation if the delta is a face in any planar embedding of $G$
    \item Given a planar graph $G$, planarity is preserved when a YY-$\Delta$ transformation is performed on $G$
    \item Given a non-planar graph, non-planarity is preserved when a $\Delta$-YY transformation is performed
\end{enumerate}
\end{prop}
\begin{proof}
Proof of 1: Consider a planar embedding of $G$ where the delta we transform is a face of the embedding.  Put the new vertex for the wye inside this face.  This gives a planar embedding of the transformed graph. 

Proof of 2: Consider a planar embedding of $G$.  We can replace the planar embedding of the wye shape with a planar embedding of the delta shape to get a planar embedding of the new graph. 

Proof of 3: Consider a non-planar graph $G$, and the graph $G'$ obtained by performing a $\Delta$-YY transformation on $G$ with the vertices of the delta being $u,v,w$.  $G$ must have a subdivision of either $K_5$ or $K_{3,3}$ since it is non-planar. If the subdivision does not involve any edges of the delta that we transform, then it still exists in $G'$, and $G'$ is non-planar. If the subdivision involves exactly one edge of the delta, say edge $uv$, then we can replace this part of the subdivision with the path $uxv$, where $x$ is the centre vertex of the new wye, which gives us another subdivision of $K_5$ or $K_{3,3}$. 

At the other extreme, suppose now that all three edges of the delta are in the subdivision. Then all three vertices of the delta must have degree greater than 2 in the subdivisions and the subdivision in $G$ must be $K_5$, as $K_{3,3}$ has no triangles. Call the vertices of the $K_5$ subdivision with degree greater than 2 $s$, $t$, $u$, $v$, and $w$. 
We had paths from $s$ and $t$ to each of $u$, $v$, and $w$ and to each other in our original subdivision, these remain unchanged after the $\Delta$-YY transformation. We also had edges $uv$, $uw$, and $vw$, these edges are deleted, we add in a new vertex $x$ and edges $\{ux, ux, vx, vx, wx, wx\}$. Now we can see that in $G'$ we have paths from each of $s$, $t$, and $x$ to each of $u$, $v$, and $w$, for a total of 9 paths. This gives us a subdivision of $K_{3,3}$ in $G'$, so $G'$ is non-planar. 

Finally, consider the case where two of the edges of the delta are in the subdivision, say $uv$ and $vw$. If $v$ only has degree 2 in the subdivision, then we could replace the path $uvw$ in the subdivision with $uw$ and follow the case above. Otherwise, $v$ has degree 3 or 4 in the subdivision, and thus represents a vertex of $K_5$ or $K_{3,3}$. We claim that upon performing a $\Delta$-YY transformation, $G'$ will have a subdivision of $K_{3,3}$. 
To prove this claim 
there are two cases, the case where $G$ has a subdivision of $K_{3,3}$ and the case where $G$ has a subdivision of $K_5$. 

Case 1: $G$ has a subdivision of $K_{3,3}$. We know that $v$ is a vertex of degree 3 in this subdivision, and we know that $u$ and $w$ are in the subdivision. Let $u'$ and $w'$ be the vertices of degree 3 in the subdivision along the paths from $v$ including $u$ and $w$ respectively and let $z$ be the third vertex of degree 3 in the subdivision which is joined to $v$.  
Note that it is possible that $u=u'$ or $w=w'$. Now perform a $\Delta$-YY transformation on the delta $\{u,v,w\}$. 
Where we previously had paths $v,\dots, w'$, $v,\dots, u'$, and $v,\dots, z$,
we now have $x,\dots, w'$, $x\dots, u'$, and $x, v, \dots, z$. The rest of the subdivision remains unchanged. Hence, we have a subdivision of $K_{3,3}$ in $G'$ and $G'$ is thus non-planar. 

Case 2: $G$ has a subdivision of $K_5$. We know that $v$ is a vertex of degree 4 in this subdivision, and we know that $u$ and $w$ are in the subdivision. Similarly to the previous case, let $u'$ and $w'$ be the vertices of degree 4 in the subdivision along the paths from $v$ containing $u$ and $w$ respectively.  Let $s$ and $t$ be the other two vertices of degree 4 in the subdivision.
Perform a $\Delta$-YY transformation on the delta $\{u,v,w\}$ with new vertex $x$.  Then $G'$ contains a subdivision of $K_{3,3}$ with $\{x,s,t\}$ and $\{v,u',w'\}$ the two parts of the bipartition,  where the $s,\dots, v$, $s, \dots, u'$, $s, \dots, w'$, $t, \dots, v$, $t, \dots, u'$, and $t, \dots w'$ paths are unchanged from the subdivision in $G$ and the remaining paths are $xv$, $xu,\dots, u'$, and $xw, \dots w'$ where the paths $u, \dots, u'$ and $w, \dots, w'$ were from the original subdivision and are thus disjoint from all aforementioned paths.

\end{proof}

\subsection{Independent Sets}
Later we will have cause to consider bipartite graphs and it will be useful to understand how independent sets transform under the $\Delta$-YY operations.
Recall that an independent set is a set of vertices in a graph $G$ where no two vertices in the set share an edge. Consequently independent sets behave under $\Delta$-YY operations as they do under $\Delta$-Y operations.  We outline this behaviour below.

\begin{lemma}\label{lem indep1}
Given a graph $G$ and an independent set $W$, $W$ is also an independent set in the graph $G'$ given by performing a $\Delta$-YY transformation on $G$.
\end{lemma}
\begin{proof}
Let $G$ be a graph and let $W$ be an independent set of $G$. 
A $\Delta$-YY transformation only removes connections between existing vertices, and does not add any connections between existing vertices. Thus, any independent set in $G$ will also be independent in $G'$.  
\end{proof}
Define $\alpha(G)$ to be the maximum size of an independent set of $G$. 
\begin{prop}
Given a graph $G$ and the graph $G'$ given by performing a $\Delta$-YY transformation on $G$, $\alpha(G)\leq \alpha(G')\leq \alpha(G)+2$.  Additionally, each bound can be attained by some $G,G'$ pairs.
\end{prop}
\begin{proof}
Let $G$ be a graph and let $W$ be a maximum independent set of $G$. We know from Lemma~\ref{lem indep1} that $W$ is an independent set in $G'$, the graph obtained by performing a $\Delta$-YY transformation on $G$. Thus $\alpha(G)\leq \alpha(G')$. Note that there do exist graphs where equality holds, in special cases when $W$ contains one vertex of the delta and a neighbour of each of the other two vertices in the delta, as for instance in a double triangle. 

Now consider a special case where there exists a maximum independent set $W$ of $G$ that contains exactly one vertex of the delta we plan to transform, say vertex $u$, and that does not contain the neighbours of $v$ or $w$, the other two vertices of the delta, except for $u$ as for example in $K_4$. 
Now when we perform the $\Delta$-YY transformation, we remove the edges $\{uv, vw, uw\}$ that connect $u$, $v$, and $w$ to each other. Since $v$ and $w$ are no longer adjacent to $u$ or each other, none of their other neighbours are in $W$, and $W$ is an independent set in $G'$, we can create a new independent set $W'=W\cup \{v,w\}$. Thus, there exist a graph $G$ and a delta where performing a $\Delta$-YY transformation on said $G$ and delta gives a new graph $G'$ where $\alpha(G')=\alpha(G)+2$. This shows that our upper bound is attained. 

Now it remains to show that there is no circumstance where $\alpha(G')\geq \alpha(G)+3$. 
Looking at the wye in $G'$, we can see that any independent set could include either the centre vertex $x$ or up to three of the external vertices $u$, $v$, and $w$. Say $W'$ is a maximum independent set of $G'$ and $x\in W'$, then the rest of $W'$ must involve vertices outside of the wye, we know that these vertices and their connections remain unchanged under the $\Delta$-YY transformation, so $W=W'\setminus\{x\}$ must be an independent set in $G$. In this case, $\alpha(G')=|W'|=|W|+1\leq \alpha(G)+1$. Now consider the case where $u,v,w\in W'$. Note that this means that none of their neighbours can be in $W'$. Again, the remainder of $W'$ must be made up of vertices in $G'$ that are not in the wye, and thus not affected by the transformation. So we know there exists an independent set $W''=W'\setminus \{u,v,w\}$ in $G$. However, since reversing the $\Delta$-YY transformation only changes the connections between $u$, $v$, and $w$, and not their connections with their neighbours, and we know that none of their neighbours are in $W'$, we find that we can add the vertex $u$ to $W''$ to get a new independent set $W$ in $G$. Thus $\alpha(G')=|W'|=|W|+2\leq \alpha(G)+2$. Similarly, taking a subset $S$ of $\{u, v, w\}$ to be in $W'$ will give an upper bound of $\alpha(G)+1$ for $|S|=2$ and $\alpha(G)$ for $|S|=1$ and $|S|=0$. Having looked at all cases, we have shown that, $\alpha(G')\leq \alpha(G)+2$. Therefore $\alpha(G)\leq \alpha(G')\leq \alpha(G)+2$.

\end{proof}
We have now seen that the $\Delta$-YY transformation preserves a lower bound on the maximal size of independent sets, and in some cases allows for even larger independent sets. The YY-$\Delta$ transformation can sometimes cause the size of a maximal independent set to decrease.
\begin{prop}
Given a graph $G$ and an independent set $W$, the graph $G'$ given by performing a YY-$\Delta$ transformation on $G$ also has $W$ as an independent set if one or fewer of the external vertices of the transformed wye are in $W$, and the centre of the transformed wye is not in $W$.
\end{prop}
\begin{proof}
Let $G$ be a graph with independent set $W$ and let $G'$ be the graph given by performing a YY-$\Delta$ transformation on $G$. Assume that one or fewer of the vertices of the transformed wye are in $W$.
Let $u$, $v$, and $w$ be the external vertices of the wye and let $x$ be the centre of the wye. Hence, $x\not \in W$ and one of the external vertices, say vertex $u$, is in $W$. Since nothing outside the wye is changed by the YY-$\Delta$ transformation, $W\setminus \{u\}$ is clearly an independent set in $G'$. Furthermore, we know that none of the neighbours of $u$ in $G$ are in $W$, and $u$ only gains the neighbors $v$ and $w$ in $G'$, neither of which are in $W$. Therefore, $W$ is an independent set in $G'$.

\end{proof}

\subsection{Connectivity}
Vertex connectivity also behaves under $\Delta$-YY operations as it does for the classical $\Delta$-Y operations, however, small vertex separations lead to identities of Feynman periods and substantially simplify Feynman period calculations, so we will lay out mathematically how connectivity behaves under the $\Delta$-YY operations.

\begin{prop}\label{prop connectivity}
Connectivity is preserved by a $\Delta$-YY operation.
\end{prop}
\begin{proof}
Both the delta and wye shapes are connected, so any vertices joined by a path using one remain joined by a path using the other, and other paths in the graph are unaffected. 
\end{proof}

\begin{prop}
2-connectivity is preserved by the YY-$\Delta$ transformation. 
\end{prop}
\begin{proof}
Consider a 2-connected graph $G$ and the graph $G'$ found by performing a YY-$\Delta$ transformation on a wye of $G$. Recall that a graph is 2-connected if and only if any two vertices of the graph lie in a common cycle. Thus, we know that this is true for $G$ and we want to show that it is true for $G'$. Consider any two vertices $a$ and $b$ of $G'$. We know that these vertices existed in $G$ and thus lie in a common cycle $C$ in $G$. If $C$ does not include any edges of our transformed wye, then the same cycle exists in $G'$ and thus $a$ and $b$ lie in a common cycle in $G'$. If $C$ does include an edge from our transformed wye, it must contain exactly two of these edges (one goes to the centre vertex and one exits the centre vertex), say that these are edges $ux$ and $vx$. Since $x$ is not in $G'$, we can replace these two edges with the new edge $uv$, to get a cycle in $G'$ that contains all the same vertices as $C$ did, except $x$, which is not in $G'$. Thus, there exists a cycle in $G'$ that contains both vertices $a$ and $b$. Therefore, $G'$ is 2-connected. 
 \end{proof} 
 \begin{prop}
 $k$-connectivity is never preserved by the $\Delta$-YY operation for $k\geq 4$. 
 \end{prop}
 \begin{proof}
 Let $k\geq 4$. Since the centre of a wye only has 3 neighbours, any graph with a wye is at most 3-connected, thus any $\Delta$-YY transformation gives a graph which is at most 3-connected.
 \end{proof}
 \begin{prop}
 The $\Delta$-YY transformation preserves 3-connectivity when each vertex of the transformed delta has at least four neighbours.
 \end{prop}
 \begin{proof}
 Let $G$ be a 3-connected graph with a delta on the vertex set $\{u, v, w\}$, call this triangle $T$. Let $G'$ be the graph obtained by performing a $\Delta$-YY transformation on $T$ in $G$ to get a wye whose centre is a vertex $x$. 
 We claim that for every pair of vertices $a$ and $b$ in $G$, there exists a set of 3 independent $a-b$ paths where at most one of them contains an edge of the triangle $T$. Assume for contradiction that there does not exist a set of 3 independent $a-b$ paths for some $a$ and $b$ where at most one of them contains an edge of the triangle $T$. We know that $G$ is 3-connected so we must still have 3 independent $a-b$ paths, and thus at least two of them contain an edge of the triangle $T$. We note that these edges must be distinct since the paths are independent, but they also share a vertex in $T$, call it $u$. Since there does not exist a set of 3 independent $a-b$ paths in $G$ containing only one edge of $T$, these edges must also be in different paths. Finally we note that these paths are independent, so they can only both contain $u$ if $u$ is either point $a$ or $b$. So let $a=u$. Since $u$ has degree 4, it has two neighbours that are not in the triangle. We know that the maximum size of a set of independent $u-b$ paths is 3, and that two of these paths must go to neighbours on $T$, so any two $u-b$ paths through neighbours of $u$ not in $T$ must share a vertex, call it $y$. Then $\{u,y\}$ is a two vertex cut, contradicting the 3-connectivity of $G$. 
 Therefore, for every pair of vertices $a$ and $b$ in $G$, there exists a set of 3 independent $a-b$ paths where at most one of them contains an edge of the triangle $T$. We aim to show that $G$ is 3-connected. We will do this by showing that for every pair of vertices $c$ and $d$ in $G$, there exist at least 3 independent paths between them. We will break this down into three cases, $c\neq x, d \neq x$; $c=x, d\notin \{u, v, w\}$; and $c=x, d\in \{u, v, w\}$. 
 
 Case 1: Neither $c$ nor $d$ are the vertex $x$. In this case, $c$ and $d$ are both in $G$, and thus there exist 3 independent $c-d$ paths in $G$. If these paths do not involve any edges of $T$, then they exist in $G'$ and we are done. Otherwise, only one of these paths may involve an edge of $T$. If this path contains only one edge of $T$, say $uv$, then we can replace $uv$ with $uxv$ in this path to get 3 independent $c-d$ paths in $G'$. If this path contains two edges of $T$, they must be consecutive edges in the path, otherwise the path would have to repeat a vertex, so say these two edges are $uv$ and $vw$. Then we can replace $uvw$ with $uxw$ in this path to get 3 independent $c-d$ paths in $G'$. The path cannot contain three edges of $T$, as this would violate the definition of a path. Hence, in this case, there exist 3 independent $c-d$ paths in $G'$ for every pair $c$ and $d$ where neither one is equal to $x$. 
 
 Case 2: One vertex, say $c$, is equal to $x$, and the other vertex $d$ is not in the vertex set $\{u, v, w\}$. We want to find 3 independent $x-d$ paths in $G'$.  We see that $x$ has exactly 3 neighbours, so each one must be the second element of one of our 3 independent paths. Thus, we can simplify this problem to finding a $u-d$ path, a $v-d$ path, and a $w-d$ path that are all independent of each other. These vertices are all in $G$, and thus in $G$ there exist 3 independent $u-d$ paths, 3 independent $v-d$ paths, and 3 independent $w-d$ paths in $G$. We note that these paths are not necessarily all independent of each other. We also note that one of these paths for each pair of vertices could contain an edge of $T$, and thus we may only have 2 of the 3 independent paths for each pair of vertices in $G'$. Consider the 3 independent $u-d$ paths, either none of these paths contain edges of $T$, or one of them does, in which case we consider the last vertex of $T$ in the path, say $v$. In this first case, we get 3 independent $u-d$ paths in $G'$. In this second case, this will give us 3 independent paths in $G'$, two from $u$ to $d$ and one from $v$ to $d$. By Menger's theorem, there exist 3 independent $x-d$ paths if the minimum size of a vertex cut separating $x$ and $d$ is at least 3. Assume for contradiction that there exists a vertex cut $W$ of size 2 that separates $x$ and $d$. We know that we have 3 $x-d$ paths in $G'$, where at least two start with $xu$ and are independent from $u$ to $d$. Thus, to separate $x$ and $d$, we must include $u$ in $W$. Separately, we know that we have two independent $v-d$ paths and two independent $w-d$ paths in $G'$. There is no single vertex where all four of these paths intersect, thus it is impossible to separate $x$ from $d$ with a vertex cut of size 2. Thus, a minimum vertex cut separating $x$ and $d$ in $G'$ has size at least 3, and by Menger's theorem, there must exist 3 independent $x-d$ paths in $G'$. 
 
 Case 3: One vertex, say $c$, is equal to $x$, and the other vertex $d$ is in the vertex set $\{u, v, w\}$. Without loss of generality, we will say that $d=u$, since $u$, $v$, and $w$ are indistinguishable. We know that $x$ has only three neighbours, so one of our three independent paths will be $xu$, which is obviously independent of any other $x-u$ path, one of our paths must start with $xv$, and the last must start with $xw$. We will look to $G$ to find a $v-u$ and a $w-u$ path in $G'$. To begin, we know that there exist 3 independent $v-u$ paths and 3 independent $w-u$ paths in $G$, where only one path from each of these sets could contain an edge of $T$. Thus, in $G'$, we have at least 2 independent $v-u$ paths, $P_{1v}$ and $P_{2v}$, and at least 2 independent $w-u$ paths, $P_{1w}$ and $P_{2w}$. Of these four paths, if there exists one $v-u$ path and one $w-u$ path that are independent of each other, then we can take these two paths with $x$ appended to the start to get 2 additional independent $x-u$ paths, for a total of 3 independent $x-u$ paths. Otherwise, both $P_{1v}$ and $P_{2v}$ must intersect with both $P_{1w}$ and $P_{2w}$. Starting from $v$, travel along $P_{1v}$ until we reach the first time it intersects with one of the $w-u$ paths, without loss of generality, assume that this first intersection occurs with the path $P_{1w}$ and call this vertex of intersection $i$. Then the paths $P_1=x\cup P_{1v}iP_{1w}$ and $P_2=x\cup P_{2w}$ are two independent $x-u$ paths. We know that $P_{1v}i$ does not intersect $P_{2w}$ because we defined $i$ to be the first time $P_{1v}$ intersected one of the $w-u$ paths. We know that $iP_{1w}$ does not intersect $P_{2w}$, except at $u$, because $i$ is some vertex after $w$, so $w\notin iP_{1w}$, and $P_{1w}$ and $P_{2w}$ only intersect at the points $w$ and $u$. Hence, $P_1$, $P_2$, and $xu$ are 3 independent $x-u$ paths. 
 
 Therefore, the $\Delta$-YY transformation preserves 3-connectivity when there exists a set of 3 independent paths in $G$ between any two vertices $a$ and $b$ where at most one of them contains an edge of the delta that is transformed.
 \end{proof}
 \begin{prop}
 The YY-$\Delta$ transformation preserves 3-connectivity. 
 \end{prop}
 This is as we would expect from quantum field theory as 3-separations lead to products of periods, see the exposition in \cite{Sphi4}.
 
 \begin{proof}
 Consider a 3-connected graph $G$ and the graph $G'$ found by performing a YY-$\Delta$ transformation on a wye of $G$. Recall that a graph is 3-connected if and only if for any two vertices of the graph, there exist 3 independent paths between them. So we know that this holds for $G$ and we want to show that it is the case for $G'$. Consider any two vertices $a$ and $b$ in $G'$, they also exist in $G$ and thus there exist at least 3 independent paths between them in $G$. If these paths do not involve any edges of the wye, they remain unaffected, and thus exist in $G'$. If any of these paths do involve edges of the wye, then they must contain exactly two of these edges (one goes to the centre vertex and one exits the centre vertex), say that these are edges $ux$ and $vx$. Since $x$ is not in $G'$, we can replace these two edges with the new edge $uv$, to get a path in $G'$ that contains all the same vertices as $C$ did, except $x$, which is not in $G'$. Hence, this set of paths is still independent in $G'$. Thus, there exist at least 3 independent paths between the vertices $a$ and $b$. Therefore, $G'$ is 3-connected. 
 \end{proof}

\subsection{Cyclomatic Number}

 The cyclomatic number is important in quantum field theory where it is known as the loop number.  The cyclomatic number transforms in a fairly straightforward way under the $\Delta$-YY operations, however it is also useful to consider the number of cycles larger than 2, that is the cyclomatic number of $G^S$ rather than of $G$ under $\Delta$-YY operations.  This is a more complicated question but the methods are quite similar to what we will use in Section~\ref{sec simple} and so it serves as a good warm up.

\begin{prop}
Let $G$ be a connected graph with cyclomatic number $c$ and let $G'$ be the graph obtained by performing a $\Delta$-YY operation on $G$ with cyclomatic number $c'$. Then,
\begin{enumerate}
    \item If we perform a $\Delta$-YY transformation to get $G'$, then $c'=c+2$
    \item If we perform a YY-$\Delta$ transformation to get $G'$, then $c'=c-2$
\end{enumerate}
\end{prop}
\begin{proof}
Let $m$ be the number of edges and $n$ be the number of vertices in $G$. Perform a $\Delta$-YY transformation to get a new graph $G'$. Recall that the cyclomatic number $c$ is given by $m-n+1$. Let $m'$ be the number of edges in $G'$ and $n'$ be the number of vertices in $G'$. We know that a $\Delta$-YY transformation adds one vertex, deletes 3 edges, and adds 6 edges, for a net total of one additional vertex and 3 additional edges. So by Euler's formula $c'=m'-n'+1=(m+3)-(n+1)+1=m-n+3=c+2$. 

The second item follows from the first by reversing the roles of $G$ and $G'$.  We stated both for the convenience of the reader.
\end{proof}

\begin{prop}\label{prop cyclomatic simple}
Let $G$ be a connected graph such that $G^S$ is a graph with cyclomatic number $c$ and let $G'$ be the graph obtained by performing a $\Delta$-YY operation on $G$, where $G'^S$ has cyclomatic number $c'$. Let $u$, $v$, and $w$ be the vertices of the delta or the external vertices of the wye that we plan to transform. Let $i$ be the number of connections between $u$, $v$, and $w$ in $G$ after we delete the edges of the delta or wye that we plan to transform. 
\begin{enumerate} 
    \item If a $\Delta$-YY operation is performed, then $c'=c-(1-i)$
    \item If a YY-$\Delta$ transformation is performed, then $c'=c+(1-i)$
\end{enumerate}
\end{prop}
\begin{proof}
By Euler's formula the dimension of the cycle space of $G^S$ is equal to $m-n+1$
where $m$ is the number of connections in $G$ and $n$ is the number of vertices. Let $m'$ be the number of connections in $G'$ and $n'$ be the number of vertices in $G'$. In the $\Delta$-YY transformation, we replace edges $\{uv, uw, vw\}$ with edges $\{ux, ux, vx, vx, wx, wx\}$. This transformation always adds 3 connections, one between $u$ and $x$, one between $v$ and $x$, and one between $w$ and $x$. This transformation always deletes 3 edges. When we remove these 3 edges, we have $i$ connections left between these vertices. So this transformation removes $3-i$ connections. Thus, $m'=m+3-(3-i)=m+i$. We notice that in all cases the $\Delta$-YY transformation adds one vertex, so $n'=n+1$. Then we can see that while $c=m-n+1\iff c-1=m-n$,
\begin{align*}
    c'&=m'-n'+1\\
    &=(m+i)-(n+1)+1\\
    &=m-n+i\\
    &=c-1+i
\end{align*}
Therefore, $c'=c-(1-i)$. 

The second item follows from the first by reversing the roles of $G$ and $G'$.

\end{proof}

\section{$\Delta$-YY Equivalence Classes of 6-regular Graphs}\label{sec equiv classes}

As with the original $\Delta$-Y operation, we say that two graphs $G$ and $H$ are \textit{$\Delta$-YY equivalent} if there exists a series of $\Delta$-YY operations that take $G$ to $H$. We often shorten this notation to say $G$ and $H$ are \textit{equivalent}. An \textit{equivalence class} of graphs is a set of graphs that are all equivalent to each other. The \textit{equivalence class of $G$} is the unique equivalence class containing $G$. 

From now on we assume our graphs are 6-regular as this is the case of primary interest for quantum field theory.

\subsection{Add-ons and excluded subgraph lemmas}\label{sec add ons}

Our first goal is to determine when these 6-regular equivalence classes are finite and when they are infinite. To do so, we must start by defining a set of excluded subgraphs. 

Our set of excluded subgraphs can be given by a combination of four base shapes and the selection of an add-on. To combine an add-on with a base shape, we must define an \textit{addition vertex} in the add-on and one or more addition vertices in the base shape. We define the addition of a graph $G$ and add-on $A$ to be the gluing of both graphs at an addition vertex. 

The add-ons can be defined as follows. An add-on is a chain of length $n\geq 0$ where each element of the chain is a 4-cycle with two adjacent connections of size 1 and two adjacent connections of size 2, and where the vertex incident to both connections of size 2 may be incident to another identical shape, continuing $n$ times and finishing with a double edge or a delta. The vertex of degree 2 incident to exactly two connections of size 1 in the first link of the chain is the addition vertex.  When $n=0$ the add-on is just a double edge or a delta and the addition vertex can be chosen arbitrarily. Figure \ref{fig:add-ons} shows some examples. The top left vertex of each graph in Figure \ref{fig:add-ons} is the addition vertex. In Figure \ref{fig:add-ons}, we can see the following add-ons:
\begin{enumerate}[a)]
    \item A chain of length 0, ending in a double edge.
    \item A chain of length 0, ending in a delta. 
    \item A chain of length 1, ending in a double edge
    \item A chain of arbitrary length, ending in a double edge. 
    \item A chain of arbitrary length, ending in a delta. 
\end{enumerate}
\begin{figure}[H]
    \centering
    \includegraphics[width=8cm]{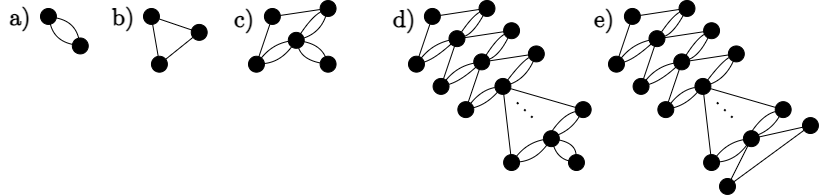}
    \caption{Examples of add-ons}
    \label{fig:add-ons}
\end{figure}

Now we can define the excluded subgraphs as the following set of graphs. Figure \ref{fig:excl_sub} has been provided for clarity, note that it does not include all instances of using add-ons. 
\begin{enumerate}
    \item A 3-cycle where at least one of the connections has size greater than 1.
    \item A 4-cycle where at least three connections of the cycle have size greater than 1 and one of the vertices incident to two connections of size greater than 1 is an addition vertex to which we add an add-on. 
    \item A 5-cycle where all connections have size greater than 1 and at least 2 non-adjacent vertices are addition vertices to which we add add-ons. 
    \item A 4-cycle, where three connections of the cycle have size greater than 1, the fourth connection of the cycle is a connection of size 1 and is contained in a triangle, and the vertices incident to the connection of size 1 are addition vertices to which we add add-ons. 
\end{enumerate}
\begin{figure}[H]
    \centering
    \includegraphics[width=12cm]{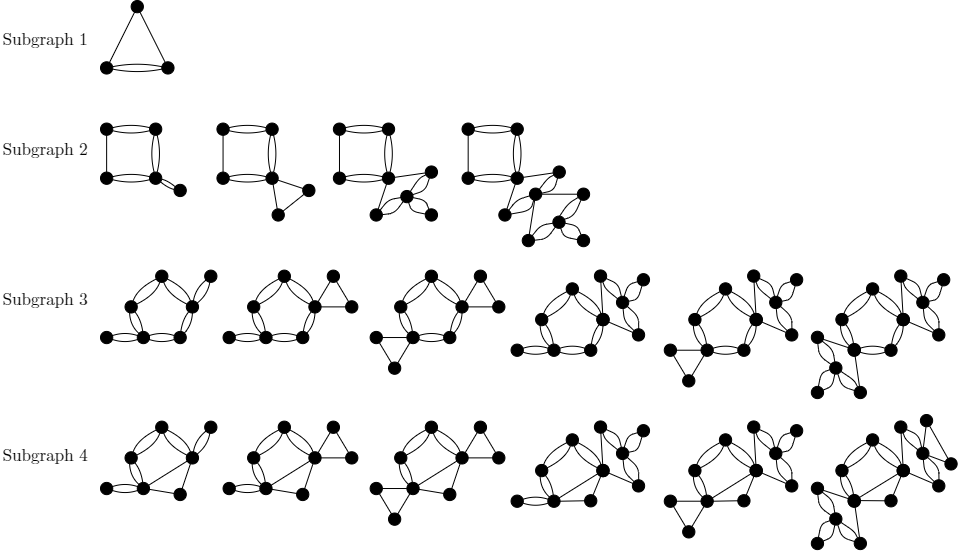}
    \caption{Some Excluded Subgraphs}
    \label{fig:excl_sub}
\end{figure}

Note that after adding the add-on, the addition vertex has degree 6, so when the excluded subgraph appears as a subgraph of a 6-regular graph then the addition vertex has no neighbours outside the excluded subgraph.  The same holds for the addition vertices of each link of the chain making up the add-on.

Given an add-on $A$ added to a base graph and all together appearing as a subgraph of a 6-regular graph, by a \emph{possible $\Delta$-YY operation on the add-on $A$}, we mean any $\Delta$-YY operation that acts on at least one edge of the add-on $A$, does not involve any edges of the base graph, but may involve any number of edges in the rest of any 6-regular graph.

\begin{lemma}\label{lem add on trans}
With set up and notation as above, any possible $\Delta$-YY operation on $A$ results in a graph containing an add-on to the same base graph or a graph containing subgraph 1.  
\end{lemma}

\begin{proof}
We will begin by considering all possible $\Delta$-YY operations that could be performed. We will look at possible transformations on a connection of size 1 of $A$, a connection of size 2 of $A$, two adjacent connections of $A$, deltas in $A$, and wyes in $A$. This analysis will thus cover all possible $\Delta$-YY operations on $A$. 

Suppose only one of the connections we are transforming is in $A$, and we are doing a $\Delta$-YY transformation.  In this case both vertices of the connection in $A$ in the graph must have degree less than 6 in the excluded subgraph in order for the other delta edges to be outside $A$.  The only way for there to be such a connection in $A$ is if $A$ ends with a delta and the connection in question is the one of the final delta which does not have a vertex in the rest of the add-on.  Performing this $\Delta$-YY transformation converts $A$ into an add-on of length one more ending with a double edge and without changing how the add-on is added to the base graph.  

Suppose only one of the connections we are transforming is in $A$ and we are doing a YY-$\Delta$ transformation.  Then the connection in $A$ is of size $2$ and at least one of its vertices must have degree less than $6$ in the excluded subgraph in order for the other wye edges to be outside $A$.  The only way for there to be such a connection in $A$ is if $A$ ends with a double edge and that double edge is the connection in question.  Performing this YY-$\Delta$ transformation converts $A$ into an add-on of the same length but ending with a delta instead of a wye and without changing how the add-on is added to the base graph.

Next suppose exactly two of the connections we are transforming are in $A$.  These two connections must share a vertex.  If we are performing a YY-$\Delta$ transformation, then we must have two connections of size 2 in $A$ which share a vertex that does not have degree 6 in the excluded subgraph (so that the third connection of the wye is not in the excluded subgraph), but no such pair of connections exists in an add-on since the shared vertex between two adjacent connections of size 2 always has degree 6 in the add on. If we are performing a $\Delta$-YY transformation and at least one of the connections has size 2 or more, then performing the transformation gives a graph containing subgraph 1.  Suppose, then, that we are performing a $\Delta$-YY transformation with two adjacent connections of size 1 in $A$ and the third edge not in $A$ and not doubling an edge in $A$.  If the two edges in $A$ are in a final triangle then the third doubles an edge of $A$, so the two must be part of one of the square links forming the add-on and the third edge then is the diagonal of this square that does not involve the addition vertices.  Performing this $\Delta$-YY transformation gives a chain formed of all the links of $A$ from before the transformed link and ending in a double edge.  This is again an add-on added to the base graph in the same way as $A$ was.

Suppose we are transforming a delta that is fully in $A$.  The only possible delta in $A$ is a final triangle, and performing a $\Delta$-YY transformation on this delta converts $A$ into an add-on of the same length but ending with a double edge and without changing how the add-on is added to the base graph.

Finally, suppose we are transforming a wye that is fully in $A$.  The only possible wye in $A$ is the one formed by a final double edge and the two double edges of the previous link, which necessarily exists.  Performing this YY-$\Delta$ transformation converts $A$ into an add-on of length one less and ending with a triangle, without changing how the add-on is added to the base graph.

This covers all cases and so proves the lemma.

\end{proof}

We have now considered every case of transforming these add-ons, and have shown that under every possible transformation, they will always be equivalent to another graph containing an add-on or subgraph 1. Next we have a quick, but extremely useful lemma.

\begin{lemma}\label{lem reduce chain}
Every add-on is $\Delta$-YY equivalent to the chain of length 0 ending in a double edge. 
\end{lemma}
\begin{proof}
Consider a chain ending in a delta, we can apply a series of alternating $\Delta$-YY and YY-$\Delta$ transformations to the chain, starting from the delta at the end of the chain and moving towards the addition vertex, to reduce to the chain of length 0 ending in a double edge. Now consider a chain ending in a double edge, if the chain has length 0, we are done. Otherwise, we can apply a series of alternating YY-$\Delta$ and $\Delta$-YY transformations to the chain, starting from the wye at the end of the chain and moving towards the addition vertex, to reduce to the chain of length 0 ending in a double edge. 
\end{proof}
This lemma allows us to reduce to the case where the chain of length 0 ending in a double edge is the only add-on when we prove that all the excluded subgraphs are equivalent to each other under every possible transformation. In this case, every possible transformation refers to all transformations that involve at least one edge of the base shape in the excluded subgraph, and any number of edges in the surrounding 6-regular graph. We do not consider transformations involving only the add-on and the surrounding graph, as we did that in Lemmas 3.1 and 3.2 to show that they are all equivalent to each other and can be reduced to the chain of length 0 ending in a double edge.

\begin{lemma}\label{lem excluded}
If a graph $G$ contains one of the excluded subgraphs, then all graphs that it is $\Delta$-YY equivalent to contain an excluded subgraph. 
\end{lemma}
\begin{proof}
If there is an excluded subgraph which the $\Delta$-YY operation does not involve any edge of, then the transformed graph still contains this excluded subgraph.  Thus we need only consider $\Delta$-YY operations which do use at least one edge from an excluded subgraph.
By Lemma~\ref{lem add on trans}, we only need to consider $\Delta$-YY operations that involve at least one edge of the base graph.

We will cover the possibilities in the same order as in the proof of Lemma~\ref{lem add on trans}.

Suppose only one of the connections we are transforming is in the base graph and we are doing a $\Delta$-YY transformation.  If the connection in the base graph has size 2 or more, then performing the transformation results in subgraph 1 which is an excluded subgraph.  Suppose then, the connection in the base graph has size 1.  Both vertices of this connection must have degree less than 6 in the excluded subgraph and so in particularly cannot be addition vertices.  Only subgraph 1 and subgraph 2 have such connections.  Performing the $\Delta$-YY on this connection in subgraph 1 results in a subgraph 2 with the same add-on, and performing it in subgraph 2 results in a subgraph 3 with the original add-on as one of its add-ons and a double edge as the other add on.

Suppose only one of the connections we are transforming is in the base graph and we are doing a YY-$\Delta$ transformation.  Then at least one of the vertices of the connection in the base graph must have degree 2 in the base graph in order for the other wye edges to be outside the base graph, but no such vertices exist in any base graph, so this case is impossible. 

Next suppose exactly two of the connections we are transforming are in the base graph.  These two connections must share a vertex.  If we are performing a $\Delta$-YY transformation and at least one of the connections has size 2 or more, then performing the transformation gives a graph containing subgraph 1.  If we are performing a $\Delta$-YY transformation with two connections of size 1 from the base graph, then these two edges are part of a triangle in the base graph and so the third edge of the delta parallels an edge of the base graph and hence is also a multiple edge. Suppose we are performing a YY-$\Delta$ transformation.  Then we must have two connections of size 2 in the base graph which share a vertex that does not have degree 6 in the base graph so that the third connection of the wye is not in the base graph.  No such pair of connections exists in a subgraph 1, but they can occur in the other base graphs.  There are two ways a pair of connections of the type we need can occur in a base graph: when the third double edge of the wye is the double edge of an add-on of size 1 ending with a double edge, or when the third double edge of the wye is not in the excluded subgraph.  In either situation, performing a YY-$\Delta$ on subgraph 2 or subgraph 4 results in a subgraph 1 and performing it on subgraph 3 results in a subgraph 2 with one of the add-ons from the subgraph 3 becoming the add-on for the subgraph 2.   

Suppose we are transforming a delta that is fully in the base graph.  Then either we are transforming a delta of subgraph 1 which results in another subgraph 1, or we are transforming the delta of subgraph 4, resulting in a subgraph 3 with the same add-ons.

Finally, there are no wyes that are fully in the base graph. 

This covers all cases and so proves the lemma.

\end{proof}

\subsection{Finite and infinite $\Delta$-YY equivalence classes}\label{sec finite infinite}

Now we want to use these lemmas to characterize finite and infinite equivalence classes of 6-regular graphs under $\Delta$-YY operations. 

\begin{theorem}\label{thm infinite}
Any 6-regular graph containing one of the excluded subgraphs is an element of an infinite $\Delta$-YY equivalence class. 
\end{theorem}
\begin{proof}
By Lemma~\ref{lem reduce chain} it suffices to consider graphs with one of the excluded subgraphs where the add-ons are all of length 0 ending in a double edge, since any other graph with an add-on is equivalent to one of these.

Notice that the existence of a $3$-cycle with at least one multiedge in a graph results in an infinite equivalence class by induction, since for any graph on $k$ vertices with a $3$-cycle with at least one multiedge, we can perform a $\Delta$-YY transformation on this $3$-cycle to get a graph on $k+1$ vertices with another such $3$-cycle. We may hence continue to apply $\Delta$-YY operations on this variety of 3-cycles an arbitrary number of times to get a graph of arbitrarily large size. 

Consider a graph $G$ containing subgraph 1. Subgraph 1 by definition is a 3-cycle with at least one multiedge, and thus $G$ is an element of an infinite equivalence class. 

Next, consider a graph $G$ containing subgraph 2. Note that since we are only taking add-ons of length 0 ending with a double edge, subgraph 2 contains a wye, one whose centre is a vertex of a four cycle. If we perform a YY-$\Delta$ transformation on this wye, we get a 3-cycle with at least one multiedge, and thus $G$ is an element of an infinite equivalence class. 

Now, consider a graph $G$ containing subgraph 3. Because of our assumption on the add-ons, subgraph 3 contains two wyes that are centred at two non-adjacent vertices of the 5-cycle. Since these centre vertices are not adjacent, we can perform a YY-$\Delta$ transformation on both of these wyes. Since a YY-$\Delta$ transformation deletes the centre vertex of the wye, this results in a 3-cycle with at least one multiedge (the edge of the original 5-cycle that did not get transformed in either of our YY-$\Delta$ transformations). Thus, $G$ is an element of an infinite equivalence class.

Finally, consider a graph $G$ containing subgraph 4. Begin by performing a $\Delta$-YY transformation on the delta in the subgraph. Now we see that our graph contains subgraph 3, which we have just shown can be reduced to a 3-cycle with at least one multiedge. Hence, $G$ is an element of an infinite equivalence class.

Therefore, any 6-regular graph containing one of the excluded subgraphs is an element of an infinite equivalence class. 
\end{proof}

\begin{theorem}\label{thm finite}
All connected 6-regular graphs that do not contain an excluded subgraph are in finite $\Delta$-YY equivalence classes. 
\end{theorem}
\begin{proof}
Assume that $G$ is a connected 6-regular graph that does not have one of the subgraphs listed above. First we claim that without loss of generality, we can assume that $G$ has no triangles.

Proof of Claim: All triangles must be single-edged triangles by our assumptions. Furthermore, if we perform a $\Delta$-YY transformation on a single-edged triangle, we only add connections of size 2 to the graph, so this operation cannot create new single-edged triangles. Finally, our graph is finite, so it has a finite number of triangles within it. Thus, we can perform $\Delta$-YY operations on triangles in the graph, until there are no more, to get an equivalent graph without triangles, that also has none of the excluded subgraphs, by Lemma~\ref{lem excluded}. 

\medskip

So assume that $G$ has no triangles, or any of the excluded subgraphs. Say $G$ has $n$ vertices. Assume for contradiction that $G$ is in an infinite equivalence class. Since there are finitely many 6-regular graphs with any fixed number of vertices, $G$ must be equivalent to a graph on more than $n$ vertices, and since the $\Delta$-YY operations change the number of vertices by 1, it must be equivalent to a graph on $n+1$ vertices. Since $G$ has no triangles, we cannot get to this equivalent graph in only one transformation, since our only possible first transformation is a YY-$\Delta$ transformation that reduces the number of vertices of $G$. 
Thus we need to make at least one more $\Delta$-YY transformation than YY-$\Delta$ transformation starting with a YY-$\Delta$ transformation to find an equivalent graph on more vertices. 

If we only do $\Delta$-YY transformations on deltas which were created directly as the delta of a YY-$\Delta$ transformation then every $\Delta$-YY transformation would require a corresponding YY-$\Delta$ transformation and so we could not have more $\Delta$-YY transformations than YY-$\Delta$ transformations.  Thus there must be at least one delta formed not directly as the delta of a YY-$\Delta$ transformation, but through some other combination of transformations.  Call such a delta a \emph{new delta}.

All edges of a new delta are connections of size 1 as otherwise the graph at that step contains a subgraph 1 and hence by Lemma~\ref{lem excluded} $G$ also contained an excluded subgraph.  Consider the graph at the step before a new delta was formed.  In order to form the new delta a transformation took place that created a connection of size 1.  It cannot have come from a $\Delta$-YY transformation because this transformation only creates connections of size 2, not connections of size 1. Thus, whenever we have a connection of size 1, it must come from a YY-$\Delta$ transformation.  So in the step immediately before the creation of the new delta we have the subgraph on the left in Figure~\ref{fig steps} where none of the connections can have larger size than indicated, and the new delta itself shares an edge with the delta created by the transformation as illustrated on the right in Figure~\ref{fig steps} where again none of the connections can have larger size than illustrated.

\begin{figure}
    \centering
    \includegraphics{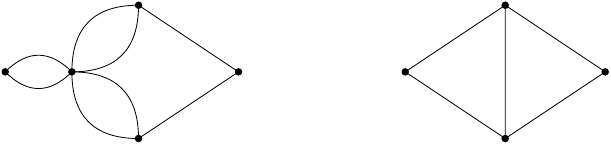}
    \caption{The subgraph just before and at the creation of the new delta.}
    \label{fig steps}
\end{figure}

However, since these two triangles share an edge, we can perform a $\Delta$-YY transformation on at most one of them in the entire sequence of transformations.  Thus this new delta does not allow for more $\Delta$-YY transformations than YY-$\Delta$ transformations.  This holds for any new delta, contradicting that $G$ is in an infinite equivalence class.

\end{proof}

\subsection{Doubled 3-regular graph families and their Feynman periods}\label{sec doubled 3 reg}

In Section~\ref{sec: Feynman Periods} we talked about our motivations for this research, $\phi^3$ scalar field theory, its Feynman graphs, and their associated Feynman periods. In that section, we noted that these Feynman graphs are 3-regular.  To apply the $\Delta$-YY operation to a Feynman graph in $\phi^3$ and obtain a period invariant, we must double all of its edges to get a 6-regular graph. In this section, we will look at the equivalence classes of these doubled 3-regular graphs, and see what they can tell us about the Feynman period.

To begin, we will look at which graphs are in infinite equivalence classes and which are not. 

\begin{cor}\label{cor finite}
A 3-regular graph will be in a finite $\Delta$-YY equivalence class once we double its edges if and only if its girth is greater than or equal to 6. 
\end{cor}
\begin{proof}
We will first prove that a 3-regular graph with girth less than 6 will be in an infinite $\Delta$-YY equivalence class once we double its edges. 

Start by noting that when considering simple graphs, which our 3-regular graphs are, there are no graphs of girth 0, 1, or 2. Thus we only need to consider graphs of girth 3, 4, and 5. In the beginning of Section~\ref{sec add ons} we saw a list of excluded subgraphs that, by Theorem~\ref{thm infinite} result in infinite equivalence classes. We will show that each graph of girth 3, 4, and 5 contains one of these subgraphs when its edges are doubled. Start with a graph of girth 3; it must have a 3-cycle.  When we double its edges, we get a 3-cycle with only connections of size 2, this contains subgraph 1. Thus, all graphs of girth 3 are in infinite equivalence classes. Next consider a graph of girth 4; it must have a 4-cycle.  When we double its edges we get a 4-cycle with only connections of size 2.  Note also that the graph was 3 regular, so when we double its edges, every vertex becomes the centre of a wye. Thus, after doubling it must contain subgraph 2 and is hence in an infinite equivalence class. Finally, consider a graph of girth 5. This graph must have a 5-cycle, when we double its edges we get a 5-cycle with only connections of size 2, where every vertex is a wye. Furthermore, this 5-cycle must have at least 2 non-adjacent vertices because for every vertex $v$ in the cycle, there are 4 other vertices in the cycle and $v$ only has 3 neighbours, so there exists at least one vertex in the cycle that it is not adjacent to. Hence, every graph of girth 5 contains subgraph 3 after doubling and thus is in an infinite equivalence class. 
Therefore, every 3-regular with girth less than 6 will be in an infinite equivalence class once we double its edges.  

\medskip
We will next prove that a 3-regular graph with girth greater than or equal to 6 will be in a finite $\Delta$-YY equivalence class once we double its edges, which will complete the proof of the corollary. 

In Theorem~\ref{thm finite} we show that all 6-regular graphs that do not contain one of the excluded subgraphs are in finite equivalence classes. We notice that all of the listed subgraphs have girth strictly less than 6 even when ignoring cycles from multiple edges, so when we double the edges of a 3-regular graph with girth greater than or equal to 6, it is impossible for it to contain one of the listed subgraphs. Therefore all 3-regular graphs with girth greater than or equal to 6 will be in finite equivalence classes once we double their edges. 
\end{proof}

Call the graphs with the minimum number of vertices in an equivalence class \textit{minimal graphs}. A \textit{minimal graph of $G$} is a minimal graph in the equivalence class of $G$. In a given equivalence class, call the minimal graphs with the minimum number of connections \textit{minimum graphs}. A \textit{minimum graph of $G$} is the minimum graph in the equivalence class of $G$.

Borinsky and Schnetz calculated the periods for all $\phi^3$ graphs up to and including loop order 6, many at loop orders 7, 8, and 9, and some particular ones at higher loop orders, all computed by graphical functions.  These periods were gratiously provided to us by Borinsky and Schnetz and will appear along with further computations in \cite{BSphi3}.  All the periods they calculated can be expressed as linear combinations of products of multiple zeta values.  

Generally speaking the complexity of the graph determines the complexity of the period, though proven rigorous statements to that effect are only known in some special cases.  In Borinsky and Schnetz' calculations, the lowest loop order $\phi^3$ graphs (that are not already determined due to small edge cuts or small vertex cuts), give the period $1$, then we see the period $\zeta(3)-\frac{1}{3}$ appearing, followed by more complicated expressions such as $-\frac{10}{3}\zeta(5)+\frac{10}{3}\zeta(3)+\frac{1}{3}$ and later expressions involving products of zeta values, for example $-3\zeta(3)^2 + 4\zeta(3)$, and eventually multiple zeta values that are not single zeta values.  Observe that the expressions are not homogeneous of weight and may include constant terms.

As the graphs get to have larger loop order, and otherwise more complex, we see new terms which can appear in the linear combination, first $\zeta(3)$, then $\zeta(5)$, $\zeta(7)$, $\zeta(3)^2$ and so on.  However, sometimes graphs have simpler periods that their size alone would suggest.  This is similar to what was observed by Broadhurst and Kreimer in the '90s concerning $\phi^4$ graphs \cite{bkphi4} and a systematic understanding is lacking now as then.  Emprically, the minimum graph in the $\Delta$-YY equivalence class, gives us some information on which period terms can appear.  This is quite reasonable as the period is a $\Delta$-YY invariant \cite{borinsky2021graphical}, and so if a large graph has a minimum graph with few vertices or few connections in its family, then the large graph is effectively much simpler than it appears.

In Table 1, we can see a table that gives us the minimum number of vertices and connections required in a minimum graph of $G$ for certain terms to appear in the period of a $\phi^3$ Feynman graph $G^S$, where $G$ is $G^S$ with all edges doubled.  This can be compared to the $\phi^4$ results of \cite{bkphi4} and \cite{Sphi4}.

\begin{table}[h]
    \centering
    \begin{tabular}{|m{1.6cm}|m{1.8cm}|m{2.1cm}|l|}
     \hline
     minimum loop number & minimum number of vertices & minimum number of connections &new period terms  \\
    \hline
     1&3 & 3& 1\\\hline
     3&5&10&\(\zeta(3)\)\\ \hline
     4&6&12& \(\zeta(5)\)\\\hline
     5&7&14&\(\zeta(7)\)\\\hline
     5&7&16&\(\zeta(3)^2\)\\\hline
     6&7&17&\(\zeta(5)\zeta(3)\)\\\hline
     6&7&19&12\(\zeta(5,3)-29\zeta(8)\)\\\hline
     6&8&16&\(\zeta(9)\)\\\hline
     6&8&17&\(\zeta(3)^3\)\\\hline
     7&9&18&\(\zeta(11)\)\\\hline
     7&9&19&5\(\zeta(6)\zeta(5)-6\zeta(4)\zeta(7)-90\zeta(2)\zeta(9)-2\zeta(5,3,3),\) \\
     &&&\(\zeta(5)^2, \zeta(3)^2\zeta(5), \zeta(7)\zeta(3)\)\\\hline
     7&9&21&94\(\zeta(7,3)-793\zeta(10), \zeta(3)\zeta(5,3), \zeta(3)\zeta(8)\)\\
     \hline
    \end{tabular}
    \caption{Minimal Conditions for New Period Terms}
    \label{tab:table 1}
\end{table}

In Table~\ref{tab:table 1}, we see that some of the period terms are listed as linear combinations of zeta values, rather than a single term. These seem to be the most convenient linear combinations by which to organize the new terms which appear.  Not all linear combinations of multiple zeta values can appear as $\phi^3$ periods, and the best basis for the purposes of writing them is not clear to us.\footnote{The study of relations between multiple zeta values is rich and conjecturally fairly well understood, though interesting questions remain.  It is an important part of the number theoretic study of multiple zeta values.}  Note also the absence of single even zetas, at least so far, similarly to the situation with $\phi^4$ periods \cite{Sphi4}.

The data in Table~\ref{tab:table 1} has been gathered up to loop 7 Feynman graphs.  We expect similar patterns to continue to hold at higher loop orders.
We also note that it is not always easy or convenient to find the minimum graph of a given graph $G$. To that end, we provide the following proposition and conjecture which should aid in providing quick, though slightly less specific, information about period terms. 

\begin{prop}\label{prop bipartite small}
A bipartite 6-regular graph $G$ on $n$ vertices obtained by doubling the edges of a 3-regular graph has minimal graphs on $\leq \frac{n}{2}$ vertices.
\end{prop}
\begin{proof}
Consider the largest independent set of $G$, this is the same as the largest independent set of $G^S$. Since $G^S$ is 3-regular and bipartite, each class of the bipartition must be of equal size. Since each class is by definition an independent set, the largest independent set of $G$ is of size at least $\frac{n}{2}$. Every vertex of $G$ is the centre of a wye, so we can perform the YY-$\Delta$ transformation on every vertex in an independent set. Each YY-$\Delta$ transformation deletes one vertex. So a minimal graph is of size $\leq n-\frac{n}{2}=\frac{n}{2}$.
\end{proof}

\begin{conjecture}\label{conj not bipartite}
A non-bipartite 6-regular graph $G$ on $n$ vertices obtained by doubling the edges of a 3-regular graph has minimal graphs on $\geq \frac{n}{2}+1$ vertices. 
\end{conjecture}

In examples up to loop 7, a non-bipartite graph has not yet been observed to have a minimal graph of size less than $\frac{n}{2}+1$, and one certainly cannot be found by exclusively performing YY-delta transformations, as there is not a large enough independent set.

\medskip

One of the most notable examples of minimal graphs giving information on the period of a graph is what we call the triple-edged triangle equivalence class. This is the class of graphs that are equivalent to a 3-cycle where every connection has size 3. The graphs in this class for which $G^S$ is a 3-regular graph all have period 1 since the triple-edged triangle has period 1.  It would be interesting to find an example of a graph with period 1 which is not from this class or to prove that no such graph can exist, and to find a more structural understanding of the equivalence class of the triple-edged triangle. These graphs up to loop 7 can be seen in Figure \ref{fig:period1}.
\begin{figure}
    \centering
    \includegraphics[width=12cm]{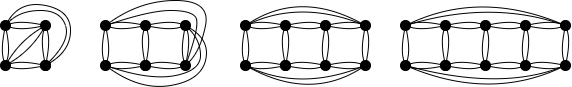}
    \caption{Graphs with Period 1}
    \label{fig:period1}
\end{figure}
It would be interesting to consider defining $\Delta$-YY reducibility as equivalence to this triple-edged triangle graph. Unfortunately, we have been unable to characterize this equivalence class, its graphs have similar and interesting properties, but we have yet to find a property which characterizes the class. 

Another property that we observed in all equivalence classes containing a double-edged graph which we have computed is that connections of size 1 always occur in cycles where each connection in the cycle is a connection of size 1. 
For some intuition, consider a double-edged graph. Performing one YY-$\Delta$ transformation gives three connections of size 1 contained in the same 3-cycle, so the property holds. From here, any YY-$\Delta$ transformation either creates another such 3-cycle, or adds one edge to an existing connection of size 1. In this second case, let our existing cycle be the triangle on vertices $\{a, b, c\}$. The new YY-$\Delta$ transformation creates a new triangle on vertices $\{b, c, d\}$. The new edge $bc$ removes one connection of size 1 from our existing cycle, while the remaining new edges $bd$ and $cd$ give us two new connections of size 1, creating a 4-cycle $\{a, b, d, c, a\}$. Similarly, performing a $\Delta$-YY transformation on the double-edged graph will create the same cycles of connections of size 1. One can proceed inductively to see that the observation is logical. 
As expected from quantum field theory we observe no connections of size greater than 2 apart from the triple-edged triangle, since such graphs do not have a convergent period.

One 3-regular Feynman graph of interest is the Mobius-Kantor graph shown in Figure \ref{fig:721MK}.  This graph is $7,21$ in the Borinsky, Schnetz list \cite{BSphi3}.
\begin{figure}[h]
    \centering
    \includegraphics[width=5cm]{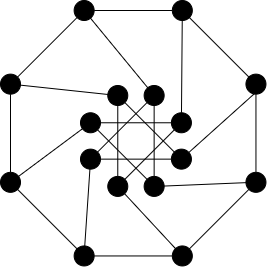}
    \caption{Graph 7,21 (drawn with simple edges): Mobius-Kantor Graph}
    \label{fig:721MK}
\end{figure}
This is a graph of girth 6, and thus belongs to a finite equivalence class. This equivalence class has a unique maximal element, which is the double-edged Mobius-Kantor graph, and a unique minimal element, which is a simple graph shown in Figure \ref{fig:721min}. 
\begin{figure}[h]
    \centering
    \includegraphics[width=4cm]{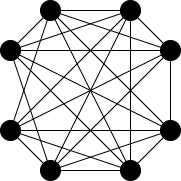}
    \caption{The minimum graph of the Mobius-Kantor graph's equivalence class}
    \label{fig:721min}
\end{figure}
Simple graphs do not appear in any infinite equivalence classes, which most of the $\phi^3$ Feynman graphs belong to, making them another interesting area of study, which we will look at in the next section. 

\subsection{Simple Graph Families}\label{sec simple}

For the graph theorists 
who deal mostly in simple graphs, and to further investigate the structure of finite equivalence classes, we also look specifically at the equivalence classes of simple 6-regular graphs. These equivalence classes have many nice properties that do not occur in the more general case. We begin with the following result.

\begin{cor}\label{cor simple finite}
All simple 6-regular graphs are in finite $\Delta$-YY equivalence classes.
\end{cor}
\begin{proof}
This follows easily from Theorem~\ref{thm finite}. We see that all the excluded subgraphs have at least one connection of size 2. Since a simple graph has no multiedges, it cannot contain any of the excluded subgraphs. Therefore, all simple 6-regular graphs are in finite equivalence classes. 
\end{proof}

The fact that these classes are finite makes them much easier to study. Many of the nice properties below seem to be more related to the finiteness of the class, rather than the existence of a simple graph in the class and thus it is possible that they could be extended to all finite classes. We do not attempt this in this paper. 

\begin{prop}\label{prop m edges}
Let $G$ be a simple 6-regular graph with $m$ edges. Then any graph in the $\Delta$-YY equivalence class of $G$ has $m$ connections. 
\end{prop}
\begin{proof}
By Corollary~\ref{cor simple finite}, the equivalence class of $G$ is finite, thus its graphs cannot contain any of the excluded subgraphs listed earlier in section 3. Notably, they cannot contain any triangles with connections of size greater than 1. Let $H$ be any graph in the equivalence class of $G$ such that it does not contain any triangles with connections of size greater than 1.  Now recall the proof of Proposition~\ref{prop cyclomatic simple}, where we showed that in the general case, if $H$ has $m$ connections, performing a $\Delta$-YY transformation on a delta of $H$ with vertices $u$, $v$, and $w$ gives a new graph $H'$ with $m'=m+i$ connections, where $i$ is the number of connections between $u$, $v$, and $w$ of size greater than 1. Since $H$ is in a finite equivalence class, there are no connections between these three vertices of size greater than 1, so $i=0$. Thus $m'=m$. Similarly, in the proof of Proposition~\ref{prop cyclomatic simple}, we showed that in the general case, if $H$ has $m$ connections, performing a YY-$\Delta$ transformation on a wye with external vertices $u$, $v$, and $w$ gives a new graph $H'$ with $m'=m-i$ connections, where $i$ is the number of connections between $u$, $v$, and $w$ in $H$.  We claim that $i=0$. Assume for contradiction that $i>0$, then there must exist at least one edge between two of these three vertices, this would give us a triangle with two connections of size 2, contradicting the finiteness of the equivalence class of $H$. Therefore, $i=0$, so $m'=m$. Therefore, every graph in the equivalence class has the same number of connections. 
Since $G$ is a simple graph, it has the same number of edges as connections. Hence, if $G$ has $m$ edges, then every graph in the equivalence class of $G$ has $m$ connections. 
\end{proof}

\begin{theorem}\label{thm minimal}
Every simple 6-regular graph is a minimal element of its $\Delta$-YY equivalence class. 
\end{theorem}
\begin{proof}
It suffices to prove the result for connected graphs.

Let $G$ be a connected simple 6-regular graph on $n$ vertices. Assume for contradiction that there exists a graph on $n-1$ vertices in the equivalence class of $G$, call this graph $H$. By Proposition~\ref{prop m edges}, we know that $H$ has $m$ connections, where $m$ is the number of edges in $G$. Since $G$ is 6-regular, $m=\frac{6n}{2}=3n$, but $H$ is also 6-regular, so $m\leq \frac{6(n-1)}{2}=3(n-1)$, which is a contradiction.
Therefore, every simple graph is a minimal element of its equivalence class. 

\end{proof}

This is another key fact that allows us nicely write out the following results on properties of the graphs in the equivalence class. 

\begin{prop}
Let $G$ be a connected simple 6-regular graph on $i$ vertices. Then any graph on $n$ vertices in the $\Delta$-YY equivalence class of $G$ has exactly $6i-3n$ connections of size exactly 1.
\end{prop}
\begin{proof}
Consider any graph in the equivalence class of $G$, it can be obtained from $G$ by a series of $s$ $\Delta$-YY transformations and $r$ YY-$\Delta$ transformations. Further, since the equivalence class is finite, every delta must consist only of connections of size exactly 1, and no wye can have edges between its outer vertices, otherwise the graph would contain an excluded subgraph and thus be in an infinite equivalent class. Thus, in every $\Delta$-YY transformation, we lose exactly three connections of size 1 and gain exactly three connections of size 2. Similarly, in every YY-$\Delta$ transformation, we lose three connections of size 2 and gain three connections of size 1. To increase the number of vertices in a graph by $k$, we must perform $k$ more $\Delta$-YY transformations than YY-delta transformations in our series of transformations. So starting from $G$, which is a minimal graph on $i$ vertices and going to a graph on $n$ vertices, $s-r=k=n-i$. Furthermore, $G$ has $\frac{6i}{2}=3i$ connections of size 1. We then perform $s$ $\Delta$-YY transformations where we lose 3 connections of size 1 and $r$ YY-$\Delta$ transformations where we gain 3 connections of size 1. So the number of connections of size 1 is equal to $3i-3s+3r=3i-3(s-r)=3i-3(n-i)=3i-3n+3i=6i-3n$. Therefore, any graph on $n$ vertices in the equivalence class of $G$ has exactly $6i-3n$ connections of size exactly 1.
\end{proof}

\begin{prop}
Let $G$ be a connected simple 6-regular graph with $i$ vertices and cyclomatic number $c$. Then the simplified version $G^S$ of any graph on $n$ vertices in the $\Delta$-YY equivalence class of $G$ has cyclomatic number $c+i-n$.
\end{prop}
\begin{proof}
We know that for any graph $G'$ in the equivalence class of $G$, its cyclomatic number is $c'=m-n+1$ where $m$ is the number of connections in $G'$ and $n$ is the number of vertices. By Proposition~\ref{prop m edges}, the number of connections is the same for all graphs in the same equivalence class. We know that the cycle space of $G$ has dimension $c=m-i+1 $, so $m=c+i-1$ for all graphs in its equivalence class. Thus, we find that $c'=c+i-1-n+1=c+i-n$. Therefore, any graph on $n$ vertices in the equivalence class of $G$ has cyclomatic number $c+i-n$.
\end{proof}

Some of these simple graph equivalence classes also include double-edged graphs, and when they do exist, these graphs also have a special role in the equivalence class. 

\begin{theorem}\label{thm 415}
Let $G$ be a 6-regular simple graph. A double-edged graph exists in the $\Delta$-YY equivalence class of $G$ if there exists a set of edge disjoint triangles in $G$ whose size is equal to the number of vertices in $G$.  
\end{theorem}
\begin{proof}
Let $G$ be a 6-regular simple graph on $n$ vertices and assume that $G$ has a set $S$ of $n$ edge-disjoint triangles. By the Handshaking Lemma, $G$ has $\frac{6n}{2}=3n$ edges. When we perform a $\Delta$-YY transformation on a triangle, we remove three connections of size 1 and replace them with three connections of size 2. Furthermore, this transformation does not affect any other triangles in $S$, each edge-disjoint triangle is independent. Each triangle has three edges, the triangles are edge-disjoint, we have $n$ of these triangles, and we can perform $\Delta$-YY transformations on all of them. So all $3n$ edges go through $\Delta$-YY transformations, all connections of size 1 are removed, and only connections of size 2 are added. Hence, the resulting graph has only connections of size 2. Thus, we have found a double-edged graph in the equivalence class of $G$. 
\end{proof}

We notice that Theorem \ref{thm 415} is only a one way result. For any readers wanting a counterexample for the converse, one is provided in Figure \ref{fig:415counter}.

\begin{figure}
    \centering
    \includegraphics[width=3cm]{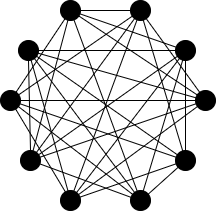}
    \caption{Counterexample for converse of Theorem \ref{thm 415}}
    \label{fig:415counter}
\end{figure}

\begin{lemma}\label{lem no big connections}
In the $\Delta$-YY equivalence class of a simple 6-regular graph, no graph in the class has a connection of size greater than 2. 
\end{lemma}
\begin{proof}
 Assume for contradiction that there exists a connection of size at least 3 between two vertices. Consider where this edge came from. It cannot be from a $\Delta$-YY transformation as that only adds connections of size 2, thus it must come from a YY-$\Delta$ transformation. A YY-$\Delta$ transformation adds one edge between each pair of its external vertices. For this to create a connection of size greater than 2, there must already be a connection of size at least 2 between two external vertices of a wye, which gives us an excluded subgraph and contradicts the fact that the equivalence class of a simple graph is finite. Therefore, there cannot exist a graph in the same equivalence class as a simple graph with a connection of size greater than 2. 
\end{proof}

\begin{theorem}
If it exists in the $\Delta$-YY equivalence class of a simple 6-regular graph, a double-edged graph is a maximal graph in the $\Delta$-YY equivalence class. 
\end{theorem}
\begin{proof}
Let $G$ be a double-edged graph on $n$ vertices in the equivalence class of a simple graph. Since $G$ is double-edged, $G^S$ is 3-regular, so we can use the handshaking lemma to find the number of connections, $m=\frac{3n}{2}$. Now assume for contradiction that there exists a graph $G'$ on $n+1$ vertices in the equivalence class. The minimum number of connections that this graph could have is $\frac{3(n+1)}{2}$, in the case where every connection is of size 2. We cannot have fewer connections because there are no connections of greater size, by Lemma~\ref{lem no big connections}. However, by Proposition~\ref{prop m edges}, $G'$ must have $m=\frac{3n}{2}$ connections. So we find that $\frac{3(n+1)}{2}\leq m = \frac{3n}{2}$, which provides a contradiction. Therefore, the double-edged graph is a maximal graph in the equivalence class. 
\end{proof}

Given these special minimal and maximal graphs, one might hope that they are unique, so that these equivalence classes form lattices, however this is not the case. The maximal double-edged graphs are not always unique, nor are the maximal graphs when double-edged graphs do not exist. In Figure \ref{fig:110} we can see a set of maximal double-edged graphs in the equivalence class of the given simple graph, showing that these are not always unique. In Figure \ref{fig:112} we can see a set of maximal graphs of the equivalence class of the given simple graph, where no double-edged graphs exist in the class.  Whether or not the simple graphs are unique minimal graphs is still an open question. We looked at all simple 6 regular graphs on 9 and 10 vertices and found that they were unique minimal graphs in their equivalence classes using a brute force method of finding equivalent graphs which starts at the simple graph and computes all possible sequences of 20 $\Delta$-YY operations. It is easy to see that any minimal graph of an equivalence class containing a simple graph must also be a simple graph, so this question is equivalent to asking if it is possible to find two equivalent simple graphs.

\begin{figure}
    \centering
    \includegraphics[width=8cm]{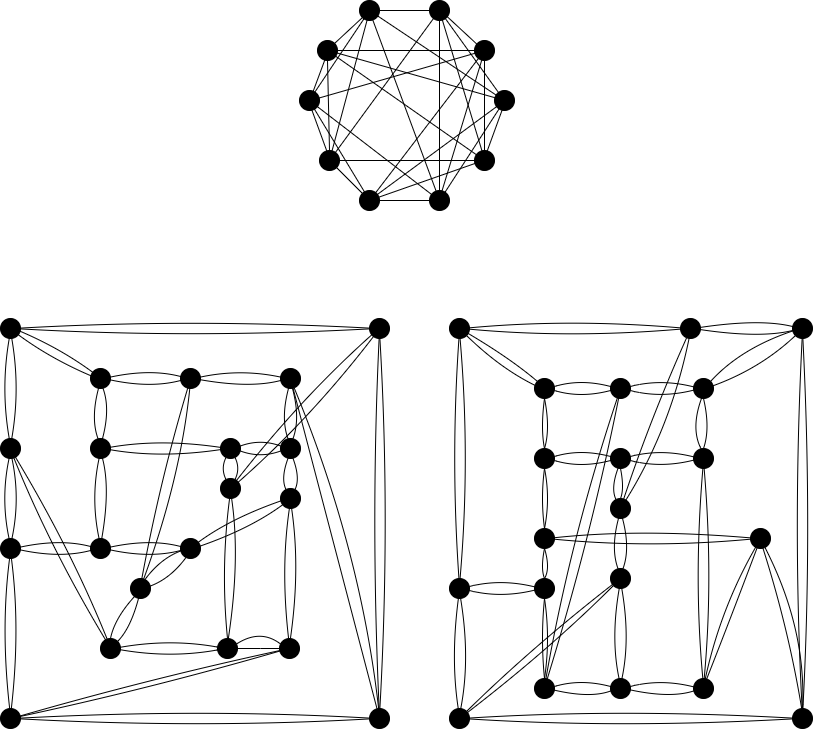}
    \caption{A simple graph and its non-unique double-edged maximal graphs}
    \label{fig:110}
\end{figure}
\begin{figure}
    \centering
    \includegraphics[width=16cm]{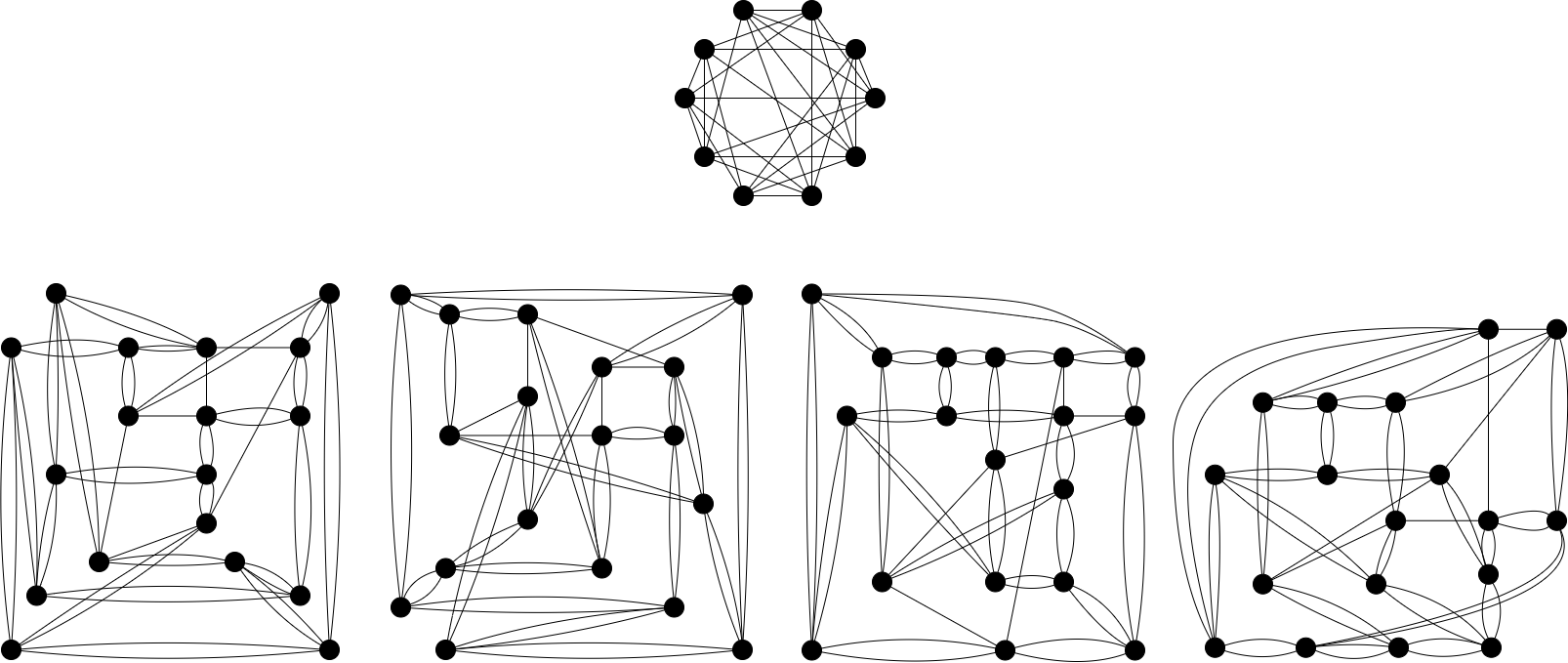}
    \caption{A simple graph and its maximal graphs}
    \label{fig:112}
\end{figure}

\section{Computational approach, observations, and discussion}\label{sec discussion}

Computations to find the data used in this paper were done by hand and using SageMath. A brute force algorithm was used to find the equivalence classes of doubled 3-regular graphs and simple graphs. Finite equivalence classes were found in full, when known to exist, while infinite equivalence classes were found by stopping the algorithm once it generated 500 to 5000 graphs. These sizes allowed us to find what should be the minimum graphs, which were used to fill out Table~\ref{tab:table 1}. 

The periods of the graphs were found by Borinsky and Schnetz \cite{BSphi3}, and we thank them for making the pre-publication results of their computations available to us. 

Let us reiterate the computational observations made in Section~\ref{sec doubled 3 reg}.  Table~\ref{tab:table 1} shows how the size of the minimal graph (measured by number of vertices or by number of connections) in the equivalence class of a doubled $\phi^3$ graph relates to the first appearance of numbers in the period of the $\phi^3$ graph as calculated by Borinsky and Schnetz \cite{BSphi3}.  This table shows that a graph having a small (by one of the above measures) member of its equivalence class is one explanation for a low weight period appearing in a higher loop order graph.  

We also investigated particular equivalence classes computationally including the Mobius-Kantor graph which has a finite equivalence class with a unique simple graph as the minimal element, and the family of of the triple triangle which we were not able to get a conclusive handle on and so must remain for future investigation.

In view of the table it is interesting to understand which doubled $\phi^3$ graphs have small minimal graphs.  We showed in Proposition~\ref{prop bipartite small} that starting from a bipartite $\phi^3$ graph gives a minimal graph with at most half the number of vertices, and we conjecture in Conjecture~\ref{conj not bipartite} that all other graphs have minimal graphs with more vertices.

Since being able to find a minimum graph of a given graph can give us information on the period of said graph, it is also useful to have an algorithm to generate minimal graphs. Unfortunately, the brute force algorithm takes a long time to run, and is not conducive to efficiently finding such minimum graphs. To that end, we believe that it would be best to use knowledge gathered to optimize this algorithm. At the moment, we do not have a proven way to do this, though we have collected a list of subgraphs that we know would allow vertices and connections to be minimized. We believe that combining a greedy strategy with these subgraphs may provide a quicker method of finding minimum graphs. For any reader interested in trying their hand at improving this algorithm, we offer the following information.

The subgraphs in Figure \ref{fig:minvert} allow us to minimize the number of vertices of a graph, this is an infinite set, though each size of subgraph has a finite number of elements. Thus, for any graph of size $n$, it can only have subgraphs of size approximately $\leq n$, giving a finite number of subgraphs for any finite graph. It is important to note that while edges must remain distinct for these subgraphs to minimize vertices, the vertices do not, so there are cases where a graph drawn with $n$ vertices could have several of these vertices be the same vertex, decreasing the number of vertices found in the subgraph. This is why the subgraphs are ordered by the number of triangles they contain, rather than their number of vertices.

\begin{figure}[h]
    \centering
    \includegraphics[width=10cm]{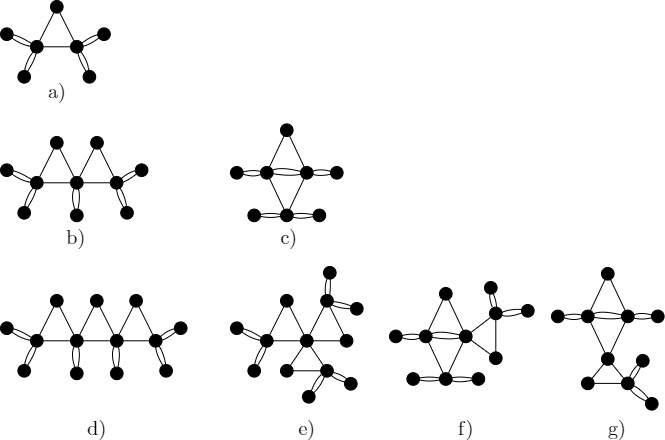}
    \caption{Subgraphs to minimize the number of vertices in a graph}
    \label{fig:minvert}
\end{figure}

While there are many of these subgraphs, we conjecture that only the graphs in the first column, and those with more triangles that follow this pattern, are needed to find a minimal graph of an equivalence class of a doubled 3-regular graph. 

\begin{conjecture}
Given a doubled 3-regular graph $G$, perform the YY-$\Delta$ transformation on all wyes in all possible ways to find a graph $G'$ of smallest size. Then $G'$ is a minimal graph iff it does not contain 
    a 
    chain of one or more 3-cycles, each connected by a vertex, which is also doubly adjacent to one more vertex, where the 3-cycles on either end of the chain have one vertex with exactly 4 neighbours and no connections of size greater than 2. 
\end{conjecture}    
    The subgraph with a length 1 chain is shown in Figure\ref{fig:minvert}a, while the next two smallest examples 
    are shown as Figures \ref{fig:minvert}b and \ref{fig:minvert}d.
The other graphs in Figure~\ref{fig:minvert} also have no YY-$\Delta$ transformations yet can be reduced to fewer vertices using a $\Delta$-YY transformation first followed by other $\Delta$-YY operations, so for arbitrary graphs, the appearance of such subgraphs also implies non-minimality.  The content of the conjecture is that for the class of a doubled 3-regular graphs, the more restricted type of subgraph described in the conjecture and illustrated in the first column of Figure~\ref{fig:minvert} is all that needs to be checked for minimality.

This conjecture is based on observations that show that all minimal graphs in equivalence classes of doubled 3-regular graphs up to loop order 6 can be found using this strategy. 

The subgraphs in Figure \ref{fig:min2} allow us to minimize the number of connections in a graph, without changing the number of vertices. This is not an extensive list and should not be taken as such, it is simply a starting point for any individual who is interested in this problem. Similarly to the subgraphs in Figure \ref{fig:minvert}, the edges of these graphs must be distinct, but the vertices are not necessarily distinct. 

\begin{figure}
    \centering
    \includegraphics[width=10cm]{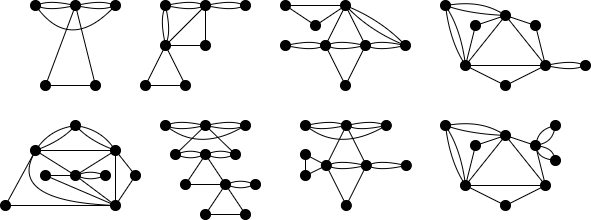}
    \caption{Subgraphs to minimize the number of connections in a graph}
    \label{fig:min2}
\end{figure}

Another interesting question is when minimal graphs are unique.  As commented at the end of Section~\ref{sec simple}, it is an open question even if simple graphs are unique minimal graphs.

\medskip

The $\Delta$-YY operation was developed in order to better understand the Feynman period of $\phi^3$ graphs so some of the most interesting questions about the operation relate to the nature of the equivalence classes of doubled 3-regular graphs -- when do they contain more than one doubled 3-regular graph, how can we understand their minimal elements, when do they contain a simple graph -- and their relation to the Feynman periods -- do the patterns in Table~\ref{tab:table 1} continue, can we predict more properties of the period from characteristics of the minimal elements or other characteristics of the equivalence classses?  In particular, is Conjecture~\ref{conj not bipartite} true?  
Relatedly, as suggested to us by Oliver Schnetz, how do $\Delta$-YY equivalence classes interact with the notion of constructibility in the theory of graphical functions, see \cite{borinsky2021graphical}.

\medskip

From a pure graph theory perspective, we have introduced a degree-preserving variant of the $\Delta$-Y operations.  This is a natural operation to play with in any context where it is desirable to preserve regularity.  Minors are no longer the correct notion for understanding reducibility, since minors are not degree preserving, but we can still ask about reducibility to a fixed graph under the $\Delta$-YY operations, or under the $\Delta$-YY operations along with some other degree preserving operations.  The family of the triple triangle seems the most promising place to start in this direction; it would be nice to understand its structure.   One could define $\Delta$-Y${}^n$ operations for $n>2$, however these operations do not preserve degree and so do not seem either as interesting or as useful.

\bibliographystyle{plain}
\bibliography{bib}

\end{document}